\documentclass{article}
\usepackage{hyperref}
\usepackage{xcolor}
\usepackage{amsmath}
\usepackage{amsthm}
\usepackage{amssymb}
\usepackage{enumitem}
\usepackage{indentfirst}

\newtheorem{theorem}{Theorem}[section]
\newtheorem*{theorem*}{Theorem}

\newtheorem{corollary}[theorem]{Corollary}

\newtheorem{lemma}[theorem]{Lemma}
\newtheorem{proposition}[theorem]{Proposition}

\newtheorem{remark}{Remark}[section]

\newcommand{\R}{\mathbb{R}}

\begin{document}
	
	\title{Critical \(p\)-biharmonic problems and applications to Hamiltonian systems}

	\author{\textbf{Kanishka Perera}\\
		\small Department of Mathematics\\\small Florida Institute of Technology\\\small
		150 W University Blvd, Melbourne, FL 32901-6975, USA\\\small
		\textit{kperera@fit.edu}\medskip\\\textbf{Bruno Ribeiro}\\
		\small Departamento de Matem\'atica\\
		\small Universidade Federal da Para\'iba\\
		\small Cidade Universit\'aria, 58051-900, João Pessoa, Brazil\\
		\small \textit{bhcr@academico.ufpb.br}
	}
	
	\maketitle

    \begin{abstract}
      We study fourth–order quasilinear elliptic problems that involve the $p$-biharmonic operator and Navier boundary conditions.  The nonlinear term grows at the critical Sobolev rate.  Starting from a Hamiltonian system of two second–order equations, we use an inversion step to reduce it to a single $p$-biharmonic equation with a lower–order perturbation.  We handle both non-resonant and resonant cases and show that the problem admits non-trivial solutions when the forcing term and the superscaled perturbation are small enough.  The proof combines concentration–compactness with an abstract critical-point method based on the cohomological index.  Our theorems cover both homogeneous and nonhomogeneous settings and extend Tarantello’s classical results for the Laplacian, improving earlier work on $p$-biharmonic equations (including the case $p=2$) and on critical Hamiltonian systems.
\medskip

    \end{abstract}

	\section{Introduction}
	
	This work aims to prove multiplicity of solutions for  Hamiltonian systems of elliptic equations in bounded domains, where one of the nonlinearities is a pure power, allowing the system to be reduced to a fourth-order \(p\)-biharmonic equation with Navier boundary conditions. More precisely, consider the system
	
	\begin{equation}\label{main}
		\left\{\begin{array}{ll}
			-\Delta u = |v|^{q-2}v & \mbox{in} \ \Omega, \\
			-\Delta v = f(u)+h(x) & \mbox{in} \ \Omega, \\
			u = v = 0 & \mbox{on} \ \partial\Omega,
		\end{array}\right.
	\end{equation}
	where \( \Omega \subset \mathbb{R}^N \) is a bounded and smooth domain, and we impose various hypotheses on \( f(s) \). We study both homogeneous and nonhomogeneous cases; thus we always assume \(f(0)=0\) and take \(h(x)\equiv0\) in the homogeneous setting. Since the nonlinearity on \( v \) can be inverted, it is well known that this system becomes equivalent to the fourth-order \( p \)-biharmonic equation given by
	
	\begin{equation}\label{mainbi}
		\left\{\begin{array}{ll}
			\Delta^2_p(u) = f(u)+h(x) & \mbox{in} \ \Omega, \\
			u = \Delta u = 0 & \mbox{on} \ \partial\Omega,
		\end{array}\right.
	\end{equation}
	where \( \Delta^2_p(u)=\Delta(|\Delta u|^{p-2}\Delta u) \) and \( p=q/(q-1) \). 
	
	\medskip
	
	Before stating the main hypotheses on \( f(s) \) and \( h(s) \), let us provide some background on the subject. There are already many results in the literature that cover this topic.
	
	\subsection{Homogeneous problem}

	Let us first focus on the case \( h(x)\equiv 0 \), that is, on the homogeneous situation. 
    There is a significant difference in the restrictions we need to impose on the dimension $N$ and the values of the exponents on the nonlinearities whether we are working with symmetric or asymmetric conditions for the Hamiltonian systems.

In 1998, Hulshof \emph{et al.}~\cite{Hulshoff-Mitidieri-Vander} studied the Hamiltonian system
\begin{equation}\label{mainEderson}
		\left\{\begin{array}{ll}
			-\Delta u = \mu|v|^{s-2}v + |v|^{q-2}v & \mbox{in} \ \Omega, \\
			-\Delta v = \lambda |u|^{r-2}u + |u|^{\tilde q-2}u & \mbox{in} \ \Omega, \\
			u = v = 0 & \mbox{on} \ \partial\Omega,
		\end{array}\right.
	\end{equation}  
    
with $N\ge 4$, $r=s=2$. 

Here, the pair \( (q,\tilde q) \) belongs to the so-called critical hyperbola, that is,
	\begin{equation}\label{criticalhyp}
		\frac{1}{q} + \frac{1}{\tilde q} = 1 - \frac{2}{N}.
	\end{equation}
This hyperbola characterizes the notion of \emph{critical growth} for Hamiltonian systems, which was independently introduced by Mitidieri~\cite{Mitidieri} and van~der~Vorst~\cite{VdVorst}, and investigated by several authors including Clément–de~Figueiredo–Mitidieri~\cite{ClementFigueiredoMitidieri} and Peletier–van~der~Vorst~\cite{PeletierVdVorst}.

In \cite{Hulshoff-Mitidieri-Vander} the authors proved the existence of positive solutions by re-parameterising the exponents \(q\) and \(\tilde q\) and adding extra conditions on \(N\).
Their results were later sharpened by dos~Santos and Guimarães~\cite{dosSantos-Guimaraes}, who assumed
\(1<r<\tilde q,\;1<s<q\) and
\((r-1)(s-1)\ge 1\).
The parameters \((\lambda,\mu)\) become decisive when \((r-1)(s-1)=1\), the ``linear'' threshold where the system couples with the spectrum of the linear operator, which generalizes the previous case $r=s=2$.
For \((r-1)(s-1)>1\) the problem is instead a superlinear, subcritical perturbation of the critical regime. They proved existence of positive solutions for small values of $\mu$ and $\lambda$ with some additional restrictions on $q$ provided $N=3$. 
	
   It is worth noting that, because the nonlinearities in~\eqref{mainEderson} are symmetric, the \emph{only} dimension that is critical to be addressed is \( N = 3 \).  This phenomenon is already visible in the classical scalar result of Brezis--Nirenberg~\cite{Brezis-Nirenberg}.

The picture changes markedly when the nonlinearities are \emph{asymmetric}.  
A key contribution in this direction is due to dos~Santos and Melo~\cite{dosSantos-Melo}, who considered the Hamiltonian system
\begin{equation}\label{mainagain}
\left\{
\begin{aligned}
-\Delta u &= |v|^{q-2}v &&\text{in } \Omega,\\
-\Delta v &= \lambda |u|^{r-2}u + |u|^{\tilde q-2}u &&\text{in } \Omega,\\
u = v &= 0 &&\text{on } \partial\Omega.
\end{aligned}
\right.
\end{equation}

By inverting the first equation they reduced the system to a single \( p \)-biharmonic problem

\begin{equation}\label{mainlufu}
		\left\{\begin{array}{ll}
			\Delta^2_p(u) = \lambda |u|^{r-2}u + |u|^{p_2^*-2}u & \mbox{in} \ \Omega, \\
			u = \Delta u = 0 & \mbox{on} \ \partial\Omega,
		\end{array}\right.
	\end{equation}  
where \( p:=q/(q-1) \) and \( p_{2}^{*}:=\tilde q = Np/(N-2p) \) (since \( (q,\tilde q) \) lies on the critical hyperbola~\eqref{criticalhyp}).  
They established the existence of a positive solution under several parameter regimes.  
The case most relevant here is \( r = p = q/(q-1) \), for which they required
\[
N \ge 6,\qquad \frac{N+\sqrt{2N}}{N-2} \le q < \frac{N}{2},\qquad \lambda < \lambda_{1},
\]
where \( \lambda_{1} \) is the first eigenvalue of \( \Delta^{2}_{p} \) on \( W^{2,p}(\Omega)\cap W^{1,p}_{0}(\Omega) \).  
Lower dimensions were also treated, provided \( r>p \), thereby excluding the \( p \)-linear case.

Our first two theorems extend their result by allowing \emph{any} \( \lambda>0 \), even at eigenvalues, and by adding both a \textit{scaled} and a \textit{superscaled} perturbation \( \lambda |u|^{p-2}u + \mu |u|^{r-2}u \), with \( p = q/(q-1) \) and \( p < r < \tilde q \).  Using the abstract theory developed in \cite{Perera4}, we prove the existence of a nontrivial solution for the Hamiltonian system given by
	
	\begin{equation}\label{mainhomfinal}
		\left\{\begin{array}{ll}
			-\Delta u = |v|^{q-2}v & \mbox{in} \ \Omega, \\
			-\Delta v = \lambda |u|^{p-2}u + \mu |u|^{r-2}u + |u|^{\tilde q-2}u  & \mbox{in} \ \Omega, \\
			u = v = 0 & \mbox{on} \ \partial\Omega,
		\end{array}\right.
	\end{equation}
	where \( (q,\tilde q) \) satisfies \eqref{criticalhyp}, \( p = q/(q-1) \) (which gives \( \tilde q = p_2^* \) and \( 1 < p < N/2 \)), and \( p < r < p_2^* \). The solution for this system is given as a solution for its equivalent fourth-order nonlinear equation given by
	
	\begin{equation}\label{mainhombifinal}
		\left\{\begin{array}{ll}
			\Delta^2_p(u) = \lambda |u|^{p-2}u + \mu |u|^{r-2}u + |u|^{p_2^*-2}u & \mbox{in} \ \Omega, \\
			u = \Delta u = 0 & \mbox{on} \ \partial\Omega.
		\end{array}\right.
	\end{equation}

These main results utilize an abstract critical point theorem based on a pseudo-index argument using the cohomological index theory (see \cite{Perera5}). Theferefore, the precise statements (Theorems~\ref{newhom1} and~\ref{newhom2}) are deferred until after the abstract framework is introduced in Section~\ref{subsec:perturbation}.

\medskip
    
   More recently and concerning multiplicity of solutions,  Lu and Fu \cite{Lu-Fu} studied the same \( p \)-biharmonic problem with critical growth and Navier boundary conditions given in \eqref{mainlufu} with $r=p$. They proved, among other results, a multiplicity theorem that guarantees the existence of multiple nontrivial solutions depending on the spectral location of the parameter \( \lambda \). Later, Manouni and Perera \cite{Manouni-Perera} considered the problem 
	\begin{equation}\label{mainbiK}
		\left\{\begin{array}{ll}
			\Delta^2_p(u) = \lambda|u|^{r-2}u + |u|^{p_2^*-2}u & \mbox{in} \ \Omega, \\
			u = |\nabla u| = 0 & \mbox{on} \ \partial\Omega,
		\end{array}\right.
	\end{equation}
	where here \( p < r < p_2^* \). The authors also proved multiplicity of solutions using a more recent abstract critical point theory given in \cite{Perera1}. Notice that this problem is actually slightly different because it concerns Dirichlet boundary conditions. Nevertheless, this is not an issue since all the techniques and results there also apply to the Navier boundary conditions (this is only a question of changing the base space from \( W_0^{2,p}(\Omega) \) to \( W^{2,p}(\Omega) \cap W_0^{1,p}(\Omega) \)).

	Notice that if we perform the inversion procedure and fix \( p = q/(q-1) \), then the problem \eqref{mainagain} becomes exactly the problems addressed in \cite{Lu-Fu}  (for \( r = p = q/(q-1) \)) or \cite{Manouni-Perera} (with Navier boundary conditions instead of Dirichlet), since \eqref{criticalhyp} gives \( \tilde q = p_2^* \). Therefore, we can state equivalent results for multiplicity of solutions to problem \eqref{mainagain} (see Theorems \ref{theotrivial1} and \ref{theotrivial2}). The main advantage of these results is that they avoid using Aubin–Talenti concentration profiles to keep the \textit{minimax} level below the compactness threshold.  Because of this, the theorems work even in low dimensions and are written without any restriction on $N$.  On the other hand, $\lambda$ has to lie sufficiently near the eigenvalues.

	\subsection{Nonhomogeneous problem}

	Now, let us give some motivations and background for the case where we suppose \(h(x)\neq 0\).  
	
	\medskip
	
	The first result worth mentioning goes back to 2004: Cao and Han \cite{Cao-Han} considered the following system
	\begin{equation*}
		\left\{\begin{array}{ll}
			-\Delta u = \mu v + |v|^{q-2}v + \delta g(x) & \mbox{in} \ \Omega, \\
			-\Delta v = \lambda u + |u|^{\tilde q-2}u + \epsilon f(x) & \mbox{in} \ \Omega, \\
			u = v = 0 & \mbox{on} \ \partial\Omega,
		\end{array}\right.
	\end{equation*}
	assuming that \( q, \tilde q > 2 \) lie on the critical hyperbola \eqref{criticalhyp}. They also assume \( f, g \in L^\infty(\Omega) \) to be nontrivial and nonnegative, since their focus is on positive solutions. In addition, they work with \( 0 < \lambda, \mu < \lambda_1 \). Their results state the existence of a positive solution for all \( 0 < \epsilon, \delta < \epsilon_0 \) and any \( N \geq 3 \), using the sub- and super-solution method. Also, for \( N \geq 4 \) they applied a dual variational method to prove the existence of a second solution. Here again, as in the homogeneous case, the presence of symmetric nonlinearities makes this problem closer to its scalar version, where we know that the critical dimension for such problems is exactly $N=3$.
	
	\medskip
	
	More closely related to our setting is the assymetric case and the result by dos Santos in 2010 \cite{dosSantos1}: Returning to the problem \eqref{main}, consider \( f(x,u) = |u|^{\tilde q - 2}u \) and \( h(x) = \epsilon g(x) \) with \( g \in C^1(\overline\Omega) \), \( g \neq 0 \). Assume that the pair \( (q, \tilde q) \) lies on the critical hyperbola \eqref{criticalhyp}. Consider also the following restrictions on \( q \):
	\begin{equation}\label{restrictq}
		\left\{\begin{array}{l}
			\displaystyle\mbox{For }N=3,4,5,6:\ q<\frac{2(N-1)}{N-2}=2^*-\frac{2}{N-2}; \vspace{2mm} \\
			
			\displaystyle\mbox{For }N=7,8,\dots,12:\ \frac{3}{2}<q<2^*-\frac{2}{N-2}\mbox{ and } \vspace{2mm}\\
			
			\displaystyle \mbox{For }N\geq13:\ \frac{3N}{2N-6}<q<2^*-\frac{2}{N-2}.
		\end{array}\right.
	\end{equation}
	
	Using the Nehari manifold technique on the associated \( p \)-biharmonic equation obtained from the inversion of the first equation (which means that \( p = q/(q-1) \)), the author proved the existence of two solutions provided \( \epsilon \) is small enough. He also proved nonexistence of nonnegative solutions if \( \epsilon \) is large enough.
	
	\medskip
	
	 Using the abstract theory developed in \cite{Perera4}, we extend these results. First, improving the upper bounds imposed on \( q \) in \eqref{restrictq} for $N\ge 6$ and also the lower bounds if $N\ge 13$ (see Corollary \ref{cor:nonhomogeneous}); second, by adding both a \textit{scaled} and a \textit{superscaled} perturbation \( \lambda |u|^{p-2}u + \mu |u|^{r-2}u \), with \( p = q/(q-1) \) and \( p < r < \tilde q \), we prove the existence of two solutions for a wider range of parameters \( \lambda \) and \( \mu \). In other words, we consider the Hamiltonian system
	
	\begin{equation}\label{mainfinal}
		\left\{\begin{array}{ll}
			-\Delta u = |v|^{q-2}v & \mbox{in} \ \Omega, \\
			-\Delta v = \lambda |u|^{p-2}u + \mu |u|^{r-2}u + |u|^{\tilde q-2}u + h(x) & \mbox{in} \ \Omega, \\
			u = v = 0 & \mbox{on} \ \partial\Omega,
		\end{array}\right.
	\end{equation}
	where \( (q,\tilde q) \) satisfies \eqref{criticalhyp}, \( p = q/(q-1) \) (which gives \( \tilde q = p_2^* \) and \( 1 < p < N/2 \)), and \( p < r < p_2^* \). Consider also its equivalent fourth-order nonlinear equation given by
	
	\begin{equation}\label{mainbifinal}
		\left\{\begin{array}{ll}
			\Delta^2_p(u) = \lambda |u|^{p-2}u + \mu |u|^{r-2}u + |u|^{p_2^*-2}u + h(x) & \mbox{in} \ \Omega, \\
			u = \Delta u = 0 & \mbox{on} \ \partial\Omega.
		\end{array}\right.
	\end{equation}
	
	We prove the existence of two solutions for these problems when \( \mu + \|h\|_{L^{(p_2^*)'}} \) is small enough, depending on the parameter \( \lambda \) and its interaction with the Dirichlet eigenvalues of \( \Delta^2_p \). As in the homogenenous case, the statements of the main results on existence and multiplicity of solutions for \eqref{mainbifinal} (and consequently for \eqref{mainfinal}) are given later (see Theorems \ref{theononhomogeneousnonressonant} and \ref{theoresonant_nonhom}), since they depend on some definition and notation introduced in the next section. 

    \medskip
    
    Although problem \eqref{mainbifinal} inspired by the Hamiltonian system \eqref{mainfinal}, it is of interest in its own right.  It can also be viewed as a nonhomogeneous critical problem for the $p$-biharmonic operator, linking our results to the classical work of Tarantello \cite{Tarantello} on the Laplacian.  Since that seminal paper, many authors have extended the theory to other operators and settings; see \cite{Perera4} and the references there for second-order operators, \cite{Guedda} for the biharmonic case (\(p=2\)), and \cite{Clapp} for general polyharmonic problems with \(p=2\).  Our results go further, adding new information even when \(p=2\), because we cover the cases \(\lambda> 0\) and \(\mu\neq 0\), which were not covered in those earlier works.
	
	\section{General critical point theory}
	
	\label{subsec:perturbation}
	
	The existence and multiplicity results proved in this paper rest on the abstract critical point framework developed in \cite{Perera4}. For completeness, we reproduce the theorem here, together with the necessary notation and definitions.
	
	\medskip
	
	Let \( (W,\lVert\cdot\rVert) \) be a uniformly convex Banach space with dual pairing \( (W^{*},\lVert\cdot\rVert_{*}) \) and duality \( \langle\cdot,\cdot\rangle \).
	We say that \( f \in C(W,W^{*}) \) is a \emph{potential operator} if there exists a
	functional \( F \in C^{1}(W,\mathbb{R}) \), called a \emph{potential} for \( f \), such
	that \( F' = f \).
	Let us consider the nonlinear operator equation
	\begin{equation}\label{eq:2.1}
		A_{p}u = \lambda B_{p}u + f(u) + \mu g(u) + h
	\end{equation}
	in \( W^{*} \), where \( A_{p}, B_{p}, f, g \in C(W,W^{*}) \) are potential operators that satisfy the assumptions listed below, \( \lambda > 0 \) and \( \mu \in \mathbb{R} \) are parameters, and \( h \in W^{*}  \).
	
	\begin{enumerate}
		\item[(A1)]
		\( A_{p} \) is \( (p-1) \)-homogeneous and odd for some \( p \in (1,\infty) \):
		\( A_{p}(tu) = |t|^{p-2}t\,A_{p}u \) for all \( u \in W \) and \( t \in \mathbb{R} \).
		\item[(A2)]
		\( \langle A_{p}u,v\rangle \le \lVert u\rVert^{p-1} \lVert v\rVert \)
		for all \( u, v \in W \), and equality holds
		if and only if \( \alpha u = \beta v \) for some \( \alpha, \beta \ge 0 \), not both zero
		(in particular, \( \langle A_{p}u,u\rangle = \lVert u\rVert^{p} \) for all \( u \in W \)).
		\item[(B1)]
		\( B_{p} \) is \( (p-1) \)-homogeneous and odd.
		\item[(B2)]
		\( \langle B_{p}u,u\rangle > 0 \) for all \( u \neq 0 \), and
		\[
		\langle B_{p}u,v\rangle
		\le \langle B_{p}u,u\rangle^{(p-1)/p}\,
		\langle B_{p}v,v\rangle^{1/p}
		\quad \text{for all } u, v \in W.
		\]
		\item[(B3)]
		\( B_{p} \) is a compact operator.
		\item[(F1)]
		The potential \( F \) of \( f \) with \( F(0) = 0 \) satisfies
		\( F(u) = o(\lVert u\rVert^{p}) \) as \( u \to 0 \).
		\item[(F2)]
		\( F(u) \ge 0 \) for all \( u \in W \).
		\item[(F3)]
		\( F \) is bounded on bounded subsets of \( W \).
		\item[(G)]
		The potential \( G \) of \( g \) with \( G(0) = 0 \) is bounded on bounded subsets of \( W \).
	\end{enumerate}
	
	\medskip
	As proved in \cite[Proposition 3.1]{Perera3}, solutions of~\eqref{eq:2.1} coincide with critical points of the
	\( C^{1} \)-functional
	\begin{equation*}
		E(u) =
		I_{p}(u) - \lambda J_{p}(u) - F(u) - \mu G(u) - \langle h, u \rangle,
		\qquad u \in W,
	\end{equation*}
	where
	\begin{equation}\label{eq:2.3}
		I_{p}(u) = \frac{1}{p} \langle A_{p}u, u \rangle,
		\qquad
		J_{p}(u) = \frac{1}{p} \langle B_{p}u, u \rangle
	\end{equation}
	are the potentials of \( A_{p} \) and \( B_{p} \), with \( I_{p}(0) = J_{p}(0) = 0 \). Note that we include the case $h=0$. This allows us to address both homogeneous and nonhomogeneous cases using this same abstract theory.
	
	\medskip
	
	The nonlinear eigenvalue problem
	\begin{equation}\label{eq:eigen}
		A_{p}u = \lambda B_{p}u
	\end{equation}
	plays a key role.  
	Let \( \mathcal{M} := \{ u \in W : I_{p}(u) = 1 \} \).
	Then \( \mathcal{M} \subset W \setminus \{0\} \) is a bounded,
	complete, symmetric \( C^{1} \)-Finsler manifold radially homeomorphic to the unit
	sphere in \( W \), and the eigenvalues of~\eqref{eq:eigen} coincide with the
	critical values of the \( C^{1} \)-functional
	\[
	\Psi(u) = J_{p}(u)^{-1}, \qquad u \in \mathcal{M}.
	\]
	Denote by \( \mathcal{F} \) the class of symmetric subsets of \( \mathcal{M} \) and by \( i(\cdot) \)
	the \( \mathbb{Z}_{2} \)-cohomological index of Fadell and Rabinowitz in \( \mathcal{F} \) (see \cite{Fadell-Rabinowitz}).
	Set \( \mathcal{F}_{k} := \{ A \in \mathcal{F} : i(A) \ge k \} \) and
	\begin{equation*}
		\lambda_{k} := \inf_{A \in \mathcal{F}_{k}} \;
		\sup_{u \in A} \Psi(u), \qquad k \in \mathbb{N}.
	\end{equation*}
	Then \( 0 < \lambda_{1} \le \lambda_{2} \le \cdots \to \infty \) is an unbounded sequence
	of eigenvalues. Notice that
	\[
	\lambda_{1} = \inf_{u \in \mathcal{M}} \Psi(u).
	\]
	Writing \( \Psi^{a} := \{ u \in \mathcal{M} : \Psi(u) \le a \} \) and
	\( \Psi_{a} := \{ u \in \mathcal{M} : \Psi(u) \ge a \} \), if \( \lambda_{k} < \lambda_{k+1} \) then one has
	\begin{equation}\label{eq:index}
		i(\Psi^{\lambda_{k}}) = i(\mathcal{M} \setminus \Psi_{\lambda_{k+1}}) = k
	\end{equation}
	and \( \Psi^{\lambda_{k}} \) contains a compact symmetric subset of index \( k \). One can find details for all these facts in \cite[Theorem 4.6]{Aga-Per-Oreg} and \cite[Theorem 1.3]{Perera3}.
	
	\medskip
	
	Assume there exists \( c^{*}_{\mu,h} > 0 \) such that \( E \) satisfies the
	Palais–Smale condition \( (\mathrm{PS})_{c} \) at all levels \( c < c^{*}_{\mu,h} \).
	Define
	\begin{equation*}
		c^{*} := \liminf_{\mu, \lVert h \rVert_{*} \to 0} c^{*}_{\mu,h},
	\end{equation*}
	and set
	\begin{equation*}
		E_{0}(u) := I_{p}(u) - \lambda J_{p}(u) - F(u), \qquad u \in W.
	\end{equation*}
	Let \( \pi_{\mathcal{M}} \colon W \setminus \{0\} \to \mathcal{M} \),
	\( u \mapsto u / I_{p}(u)^{1/p} \), be the radial projection onto \( \mathcal{M} \).
	
	\begin{theorem}\cite[Theorem 2.1]{Perera4}\label{thm:2.1}
		Let \( \lambda_{k} \le \lambda < \lambda_{k+1} \).
		Assume there exist \( R > 0 \) and, for every sufficiently small \( \delta > 0 \),
		a compact symmetric subset \( C_{\delta} \subset \Psi^{\lambda + \delta} \) with
		\( i(C_{\delta}) = k \) and an element \( w_{\delta} \in \mathcal{M} \setminus C_{\delta} \) such
		that, putting
		\[
		A_{\delta} := \left\{
		\pi_{\mathcal{M}}\bigl((1 - \tau)v + \tau w_{\delta}\bigr)
		: v \in C_{\delta},\; 0 \le \tau \le 1
		\right\},
		\]
		one has
		\begin{align}
			\sup_{u \in A_{\delta}} E_{0}(Ru) &\le 0, \label{eq:2.8}\\
			\sup_{u \in A_{\delta},\,0 \le t \le R} E_{0}(tu) &< c^{*}. \label{eq:2.9}
		\end{align}
		Then there exists \( \mu_{0} > 0 \) such that
		equation~\eqref{eq:2.1} possesses two distinct solutions
		\( u_{1}, u_{2} \) satisfying
		\begin{equation*}
			E(u_{1}) < E(u_{2}), \qquad
			0 < E(u_{2}) < c^{*}_{\mu,h},
		\end{equation*}
		for all \( \mu \in \mathbb{R} \) and \( h \in W^{*} \) with
		\( \lvert \mu \rvert + \lVert h \rVert_{*} < \mu_{0} \).
	\end{theorem}
	\begin{remark}
In this Theorem we allow \(h\) to be zero, because we will apply the abstract result to \emph{both} the homogeneous problems (Theorems~\ref{newhom1} and~\ref{newhom2}) and the non-homogeneous ones (Theorems~\ref{theononhomogeneousnonressonant} and~\ref{theoresonant_nonhom}).  
Theorem~2.1 in \cite{Perera4} is written only for the case \(h\neq0\); that paper is concerned with obtaining \emph{two} non-trivial solutions for elliptic equations with non-zero data.  
If we set \(h = 0\), the same proof still yields two critical points \(u_1\) and \(u_2\).  
The level estimate \(E(u_2)>0\) ensures \(u_2\not\equiv0\); however, \(u_1\) might be the trivial solution.  
When \(h\neq0\) the original conclusion holds and both \(u_1\) and \(u_2\) are non-trivial.
\end{remark}

	\subsection{Eigenvalue problem}
	
	The abstract eigenvalue problem \eqref{eq:eigen} fits into the framework we use in this work. More precisely, let
	\(W=W^{2,p}(\Omega)\cap W_{0}^{1,p}(\Omega)\) and consider the potential operators
	\(A_{p},B_{p}\in C\!\bigl(W,W^{*}\bigr)\) given by
	\[
	\langle A_{p}u,v\rangle =\int_{\Omega}\!|\Delta u|^{p-2}\,\Delta u\,\Delta v\,dx,
	\qquad
	\langle B_{p}u,v\rangle =\int_{\Omega}\!|u|^{p-2}\,u\,v\,dx .
	\]
	
	Indeed, (A\textsubscript{1}) and (B\textsubscript{1}) are immediate;
	(B\textsubscript{2}) follows from Hölder’s inequality; and
	(B\textsubscript{3}) holds because
	\(W^{2,p}(\Omega)\cap W_{0}^{1,p}(\Omega)\hookrightarrow\hookrightarrow L^{p}(\Omega)\).
	For (A\textsubscript{2}) we have
	\[
	\langle A_{p}u,v\rangle
	\le \int_{\Omega}|\Delta u|^{p-1}|\Delta v|\,dx
	\le \Bigl(\int_{\Omega}|\Delta u|^{p}\,dx\Bigr)^{(p-1)/p}
	\Bigl(\int_{\Omega}|\Delta v|^{p}\,dx\Bigr)^{1/p}
	=\lVert u\rVert^{p-1}\lVert v\rVert ,
	\]
	for all \(u,v\in W^{2,p}(\Omega)\cap W_{0}^{1,p}(\Omega)\);
	equality holds throughout if and only if \(\alpha u=\beta v\) for some
	\(\alpha,\beta\ge 0\), not both zero.
	
	\medskip
	
	Now, consider the eigenvalue problem
	
	\begin{equation}\label{maineigenvalueproblem}
		\left\{\begin{array}{ll}
			\Delta^2_p(u) = \lambda|u|^{p-2}u& \mbox{in} \ \Omega, \\
			u = \Delta u = 0 & \mbox{on} \ \partial\Omega.
		\end{array}\right.
	\end{equation}
	
	Thus, according to the previous discussion, the following result for problem \eqref{maineigenvalueproblem} holds true.
	
	\begin{theorem}\label{theoeigensequence}
		Problem \eqref{maineigenvalueproblem} possesses an unbounded sequence of positive
		eigenvalues
		\(\lambda_{1}\le\lambda_{2}\le\cdots\)
		with the property that if \(\lambda_{m}<\lambda_{m+1}\) then the unit
		sphere
		\[
		S=\bigl\{u\in W^{2,p}(\Omega)\cap W_{0}^{1,p}(\Omega):\lVert u\rVert=1\bigr\}
		\]
		contains a compact symmetric subset \(C\) of cohomological index \(m\)
		such that
		\[
		\int_{\Omega}|u|^{p}\,dx\;\ge\;\frac1{\lambda_{m}}
		\qquad\text{for every }u\in C.
		\]
	\end{theorem}

\subsection{Variational framework for the \(p\)-biharmonic problem}
Let us define \( W := W^{2,p}(\Omega) \cap W^{1,p}_{0}(\Omega) \), endowed with the usual norm
	\[
	\|u\| = \left( \int_{\Omega} |\Delta u|^{p}\,dx \right)^{\frac{1}{p}}.
	\]
	This space is the environment where we are going to apply the abstract critical point theory given in Section~\ref{subsec:perturbation}. For \( u \in W \), consider the \( C^{1} \) functional
	\begin{equation}\label{functionalE}
		E(u)
		:= \frac{1}{p} \int_{\Omega} |\Delta u|^{p}\,dx
		- \frac{\lambda}{p} \int_{\Omega} |u|^{p}\,dx
		- \frac{\mu}{r} \int_{\Omega} |u|^{r}\,dx
		- \frac{1}{p_{2}^{*}} \int_{\Omega} |u|^{p_{2}^{*}}\,dx
		- \int_{\Omega} h\,u\,dx.
	\end{equation}
	With derivative \( E^\prime(u) \) given by
	\begin{equation*}
		E^\prime(u)v
		= \int_{\Omega} |\Delta u|^{p-2} \Delta u \Delta v\,dx
		- \lambda \int_{\Omega} |u|^{p-2} u\,v\,dx
		- \mu \int_{\Omega} |u|^{r-2} u\,v\,dx
		- \int_{\Omega} |u|^{p_{2}^{*}-2} u\,v\,dx
		- \int_{\Omega} h\,v\,dx.
	\end{equation*}
	Therefore, critical points for \( E \) are the weak solutions of \eqref{mainhombifinal} if $h=0$ or \eqref{mainbifinal} if $h\neq 0$.
	
	Denote by \( \|u\|_\gamma \) the norm of \( u \) in the Lebesgue space \( L^{\gamma}(\Omega) \). 

    Let us also define
	\begin{equation}\label{S2p}
		S_{2,p}:=\inf_{u\in D^{2,p}(\mathbb{R}^N)\setminus\{0\}}
		\frac{\displaystyle\int_{\mathbb{R}^N}|\Delta u|^{p}\,dx}{\left(\displaystyle\int_{\mathbb{R}^N}|u|^{p_{2}^{*}}\,dx\right)^{p/p_{2}^{*}}}.
	\end{equation}
	
	Then, \( S_{2,p} \) denotes the sharp constant associated with the Sobolev embedding  
	\[
	D^{2,p}(\mathbb{R}^N) \hookrightarrow L^{p_{2}^{*}}(\mathbb{R}^N), 
	\qquad \text{where } p_{2}^{*} := \frac{Np}{N - 2p}.
	\]
	The space \( D^{2,p}(\mathbb{R}^N) \) is understood as the closure of 
	\( C_{0}^{\infty}(\mathbb{R}^N) \) with respect to the norm  
	\[
	\|u\| := \left( \int_{\mathbb{R}^N} |\Delta u|^{p}\,dx \right)^{1/p}.
	\]

    We are now ready to state the main results of this paper. 
	
	\section{Main results}
    
	In this section we state the theorems proved in the paper.  We split the presentation into two parts: one for the \emph{homogeneous} problem ($h=0$) and one for the \emph{non-homogeneous} problem ($h\neq0$).  In both situations the existence results are obtained by applying the abstract critical-point theorem recalled in Theorem~\ref{thm:2.1}.  We also give two multiplicity results for the Hamiltonian system \eqref{mainagain}; they follow from known theorems for the associated $p$-biharmonic equation.  The passage from the single fourth-order equation to the system   is justified by standard regularity arguments available in the literature.
    
	\subsection{Main results for the homogeneous case}
    
The next three results complement the work of \cite{dosSantos-Melo}. The first one guarantees a non-trivial solution of~\eqref{mainhombifinal} for every \(\lambda>0\) that is \emph{not} an eigenvalue in the sequence given by Theorem~\ref{theoeigensequence}.

	\begin{theorem}[Non-resonant, homogeneous case] \label{newhom1}
Let $\Omega\subset\R^{N}$ be a smooth bounded domain with $N\ge 6$ and let
\[
  \frac{N}{N-2} \;<\; p \;\le\; \sqrt{\frac N2},
  \qquad
  p<r<p_{2}^{*}:=\frac{Np}{N-2p}.
\]
Assume that $\lambda>0$ is \emph{not} an eigenvalue of problem~\eqref{maineigenvalueproblem} (see Theorem~\ref{theoeigensequence}) and let $\mu\in\R$.  

\smallskip
Then there exists $\mu_{0}>0$ such that, for every
\(|\mu|<\mu_{0},\)
the problem~\eqref{mainhombifinal} possesses a nontrivial  solution
$u_{0}\in W^{2,p}(\Omega)\cap W^{1,p}_{0}(\Omega)$ satisfying
\[
   0<E(u_{0})<\frac{2}{N}\,S_{2,p}^{\,N/2p},
\]
where $E$ is the variational functional defined in~\eqref{functionalE}, with $h=0$, and
$S_{2,p}$ is the optimal Sobolev constant from~\eqref{S2p}. If $p=\frac{N}{N-2}$ then, there exists $\lambda_0>0$ such that the same result holds for all $\lambda\geq \lambda_0$ which is not an eigenvalue. 
	    
	\end{theorem} 

The next theorem treats the \emph{resonant} situation, that is, when $\lambda=\lambda_k$ for some eigenvalue in Theorem~\ref{theoeigensequence}.
Our method requires a few extra assumptions; in particular, it works only in dimensions $N\ge 8$. Before stating the result, let us define the following polynomial
\[
   G_{N}(p)\;:=\;4p^{3}+(2N-8)p^{2}-N^{2}p+N^{2},
\]
For $N\ge 8$, let 
\begin{equation} \label{defofp0}
    p_{0}(N)\in\Bigl(\tfrac{N}{\,N-2\,},\,\tfrac{2N-2}{N}\Bigr)
\end{equation}
be the \emph{unique} root of $G_{N}(p)$ in this interval.
    
    \begin{theorem}[Resonant, homogeneous case]\label{newhom2}
Let $\Omega\subset\R^{N}$ be a smooth bounded domain and assume
\(N\ge 8.\)  Assume that
\[
   p_{0}(N) \;<\; p \;<\; \frac{N}{\sqrt{2(N+2)}} ,
   \qquad
   p<r<p_{2}^{*}:=\frac{Np}{N-2p}.
\]
where $p_0(N)$ is given in \eqref{defofp0}.
\smallskip
Let $\lambda=\lambda_{k}$ be an eigenvalue of problem~\eqref{maineigenvalueproblem} given in Theorem~\ref{theoeigensequence} and let $\mu\in\R$.  
Then there exists $\mu_{0}>0$ such that, for every
\( |\mu|<\mu_{0},\)
the problem~\eqref{mainbifinal} admits a nontrivial  solution
$u_{0}\in W^{2,p}(\Omega)\cap W^{1,p}_{0}(\Omega)$ satisfying
\[
   0<E(u_{0})<\frac{2}{N}\,S_{2,p}^{\,N/2p},
\]
where $E$ is given in~\eqref{functionalE} with $h=0$ and $S_{2,p}$ in~\eqref{S2p}. 
    \end{theorem}

Putting these two results together, we have the following result of existence of solution to the System \eqref{mainhomfinal}:
   
    \begin{corollary}\label{corollaryhom} Let $\Omega\subset\mathbb{R}^{N}$ be smooth and bounded.
Assume the pair $(q,\tilde q)$ satisfies~\eqref{criticalhyp},
set $p:=\dfrac{q}{q-1}$ and choose any $r$ with $p<r<\tilde q$.
Fix $\mu\in\mathbb{R}$ and $\lambda>0$.
    	\begin{enumerate}[label=\textup{(\roman*)}]
\item \emph{Non-resonant case.}
      Suppose $\lambda$ is \emph{not} one of the eigenvalues listed in
      Theorem~\ref{theoeigensequence}.
      If $N\ge 6$ and
      \[
        \frac{N+\sqrt{2N}}{N-2}\;\le q\;<\;\frac{N}{2},
      \]
      then there exists $\mu_{0}>0$ such that, whenever
      \(|\mu|<\mu_{0}, \) problem~\eqref{mainhomfinal} possesses a classical,
      non-trivial solution.
      In the borderline case $q=N/2$ the same conclusion holds whenever $\lambda\ge\lambda_{0}$ for some $\lambda_{0}>0$.

\item \emph{Resonant case.}
      Assume $\lambda=\lambda_{k}$ for some eigenvalue $\lambda_{k}$
      of Theorem~\ref{theoeigensequence}.
      If $N\ge 8$ and
      \[
        \frac{N}{\,N-\sqrt{2(N+2)}\,}
        \;<\; q \;<\;
        \frac{p_{0}(N)}{p_{0}(N)-1},
      \]
      where $p_{0}(N)$ is defined in \eqref{defofp0},
      then the same  condition on $\mu$ as above guarantees
      a classical non-trivial solution to~\eqref{mainhomfinal}.
\end{enumerate}
\end{corollary}
\begin{proof}
    One only needs to notice that solutions of \eqref{mainhombifinal} give solutions to \eqref{mainhomfinal}. For that, regularity results can be used. Since \( \Omega \) is smooth and the nonlinearites of the problem are sufficiently regular, solutions of \eqref{mainhombifinal} are \( C^{2,\gamma}(\overline\Omega) \) and putting \( v = |\Delta u|^{p-2}(-\Delta u) \), one can apply \cite[Lemma 3.2]{dosSantos} (see also \cite[Section 3]{Hulshoff-Vander}, \cite[Lemma 1]{dosSantos-Melo} and \cite[Theorem 2.3]{doO-Macedo-Ribeiro}) to see that \( v \in C^{2,\gamma}(\overline\Omega) \) and the pair \( (u,v) \) is a classical solution of \eqref{mainhombifinal}. 
\end{proof}

	 The next two results refer to the multiplicity of solutions to problem \eqref{mainagain}. As mentioned in the Introduction, the main point is that their proofs do not use the concentration profiles of Aubin-Talenti type and therefore there is no restriction on the dimension. The price to pay is that we need to get $\lambda$ close to the eigenvalues. 
	
	\begin{theorem}(Corollary of \cite[Theorem 1.2]{Lu-Fu})\label{theotrivial1}
		Let $N\geq 3$. Suppose that the pair \( (q,\tilde q) \) verify \eqref{criticalhyp} and consider \( p = q/(q-1) \) (then \( p_2^* = \tilde q \)). Let \( (\lambda_m) \) be the sequence of eigenvalues of \eqref{maineigenvalueproblem} given in Theorem \ref{theoeigensequence}. Then:
		\begin{enumerate}
			\item[(i)] If \( \displaystyle \lambda_1 - \frac{S_{2,p}}{|\Omega|^{p/N}} < \lambda < \lambda_1 \), then problem~\eqref{mainagain} with \( r = p = q/(q-1) \)
			has a pair of nontrivial solutions \( \pm u_\lambda \) such that \( u_\lambda \to 0 \) in \( W^{2,p}(\Omega) \cap W^{1,p}_0(\Omega) \) as \( \lambda \to \lambda_1 \).
			
			\item[(ii)] If for some \( m,k \in \mathbb{N} \) the spectrum satisfies
			\[
			\lambda_m \leq \lambda < \lambda_{m+1} = \cdots = \lambda_{m+k} < \lambda_{m+k+1}, \quad \text{and } \lambda > \lambda_{m+1} - \frac{S_{2,p}}{|\Omega|^{p/N}},
			\]
			then the problem has \( k \) distinct pairs of nontrivial solutions \( \pm u^\lambda_j \) (for \( j = 1,\dots,k \)) such that \( u^\lambda_j \to 0 \) as \( \lambda \to \lambda_{m+1} \).
		\end{enumerate}
	\end{theorem}
	
	\begin{proof}
		The pair \( (q,\tilde q) \) satisfies \eqref{criticalhyp} and therefore the choice of \( p = q/(q-1) \) makes the problem \eqref{mainlufu} be exactly the one addressed in \cite[Theorem 1.2]{Lu-Fu}. Therefore, one only needs to notice that solutions of \eqref{mainlufu} give solutions to \eqref{mainagain}. For that, regularity results can be used. Since \( \Omega \) is smooth, solutions of \eqref{mainlufu} are \( C^{2,\gamma}(\overline\Omega) \) and putting \( v = |\Delta u|^{p-2}(-\Delta u) \), one can apply \cite[Lemma 3.2]{dosSantos} (see also \cite[Section 3]{Hulshoff-Vander}, \cite[Lemma 1]{dosSantos-Melo} and \cite[Theorem 2.3]{doO-Macedo-Ribeiro}) to see that \( v \in C^{2,\gamma}(\overline\Omega) \) and the pair \( (u,v) \) is a classical solution of \eqref{mainagain}. 
	\end{proof}
	
	\begin{theorem}\label{theotrivial2}(Corollary of \cite[Theorem 1.1]{Manouni-Perera}) 
		Let $N\geq 3$. Suppose that the pair \( (q,\tilde q) \) verify \eqref{criticalhyp} and consider \( p = q/(q-1) \) (then \( p_2^* = \tilde q \)). Let \( p < r < p_{2}^{*} \) and let \( (\lambda_m) \) be the sequence of eigenvalues of \eqref{maineigenvalueproblem} given in Theorem \ref{theoeigensequence}. Then problem~\eqref{mainagain} possesses \( m \) distinct pairs of non-trivial
		solutions
		\[
		\pm u_{1}^{\lambda},\; \dots,\; \pm u_{m}^{\lambda}
		\]
		whenever
		\begin{equation}\label{eq:lambda-cond}
			\lambda
			> 
			r\,|\Omega|^{\,\frac{r}{p}-1}\;
			\sup_{\tau > 0}
			\left[
			\frac{\lambda_{m}}{p\,\tau^{\,r-p}}
			- \frac{2\,S_{2,p}^{\,N/2p}}{N\,\tau^{\,r}}
			- \frac{\tau^{\,p_{2}^{*}-r}}
			{p_{2}^{*}\,|\Omega|^{\,2p/(N-2p)}}
			\right],
		\end{equation}  
		In particular, the number of solutions grows without bound as
		\( \lambda \to \infty \).
	\end{theorem}
	
	\begin{proof}
		We follow the same arguments used in the proof of Theorem \ref{theotrivial1}. The difference now is that the \( p \)-biharmonic problem corresponding to \eqref{mainagain} is given by \eqref{mainbiK}, and its analysis is based on \cite[Theorem 1.1]{Manouni-Perera}. It is worth noting that the original problem in that theorem is formulated with Dirichlet boundary conditions. However, the results still apply to the case of Navier boundary conditions, since the change only affects the choice of the function space: from \( W_0^{2,p}(\Omega) \) to \( W^{2,p}(\Omega) \cap W_0^{1,p}(\Omega) \).
	\end{proof}
	
	\begin{remark}
		Since \(p<r<p_{2}^{*}\) we have
		\[
		\frac{\lambda_{m}}{p\,\tau^{\,r-p}}
		-\frac{2\,S_{2,p}^{\,N/2p}}{N\,\tau^{\,r}}
		-\frac{\tau^{\,p_{2}^{*}-r}}
		{p_{2}^{*}\,|\Omega|^{\,2p/(N-2p)}}
		\;\longrightarrow\;-\infty
		\quad\text{as }\tau\to 0 \text{ or } \tau\to\infty,
		\]
		so the supremum in \eqref{eq:lambda-cond} is finite.  
		Because \(\lambda_{m}\to\infty\) as \(m\to\infty\) this supremum
		also diverges to \( +\infty \); hence it is positive for all
		sufficiently large \( m \).
	\end{remark}

	\subsection{Main results for the nonhomogeneous problems}

The last three main results of this paper deal with the nonhomogeneous cases and are stated below.

	\begin{theorem}[Non-resonant, nonhomogeneous case]\label{theononhomogeneousnonressonant}
Let $\Omega\subset\R^{N}$ be a smooth bounded domain with $N\ge 6$ and let
\[
  \frac{N}{N-2} \;<\; p \;\le\; \sqrt{\frac N2},
  \qquad
  p<r<p_{2}^{*}:=\frac{Np}{N-2p}.
\]
Assume that $\lambda>0$ is \emph{not} an eigenvalue of problem~\eqref{maineigenvalueproblem} (see Theorem~\ref{theoeigensequence}) and let $\mu\in\R$.  

\smallskip
Then there exists $\mu_{0}>0$ such that, for every
\[
   |\mu|+\|h\|_{(p_{2}^{*})'}<\mu_{0},
   \qquad
   h\in L^{(p_{2}^{*})'}(\Omega)\setminus\{0\},
\]
the problem~\eqref{mainbifinal} possesses two distinct non-trivial solutions
$u_{1},u_{2}\in W^{2,p}(\Omega)\cap W^{1,p}_{0}(\Omega)$ satisfying
\[
   E(u_{1})<E(u_{2})\quad\text{and}\quad
   0<E(u_{2})<\frac{2}{N}\,S_{2,p}^{\,N/2p},
\]
where $E$ is the variational functional defined in~\eqref{functionalE} and
$S_{2,p}$ is the optimal Sobolev constant from~\eqref{S2p}. If $p=\frac{N}{N-2}$ then, there exists $\lambda_0>0$ such that the same result holds for all $\lambda\geq \lambda_0$ which is not an eigenvalue. 
\end{theorem}

   \begin{theorem}[Resonant, nonhomogeneous case]\label{theoresonant_nonhom}
Let $\Omega\subset\R^{N}$ be a smooth bounded domain and assume
\(N\ge 8.\) Assume that
\[
   p_{0}(N) \;<\; p \;<\; \frac{N}{\sqrt{2(N+2)}} ,
   \qquad
   p<r<p_{2}^{*}:=\frac{Np}{N-2p}.
\]
where $p_0(N)$ is defined in \eqref{defofp0}.  
\smallskip
Let $\lambda=\lambda_{k}$ be an eigenvalue of problem~\eqref{maineigenvalueproblem} given in Theorem~\ref{theoeigensequence} and let $\mu\in\R$.  
Then there exists $\mu_{0}>0$ such that, for every
\[
   |\mu|+\|h\|_{(p_{2}^{*})'}<\mu_{0},
   \qquad
   h\in L^{(p_{2}^{*})'}(\Omega)\setminus\{0\},
\]
the problem~\eqref{mainbifinal} admits two distinct non-trivial solutions
$u_{1},u_{2}\in W^{2,p}(\Omega)\cap W^{1,p}_{0}(\Omega)$ satisfying
\[
   E(u_{1})<E(u_{2})\quad\text{and}\quad
   0<E(u_{2})<\frac{2}{N}\,S_{2,p}^{\,N/2p},
\]
where $E$ is given in~\eqref{functionalE} and $S_{2,p}$ in~\eqref{S2p}.
\end{theorem}

	As a consequence, we have the following equivalent result for the Hamiltonian system \eqref{mainfinal}.

\begin{corollary}\label{cor:nonhomogeneous}
Let $\Omega\subset\mathbb{R}^{N}$ be smooth and bounded.
Assume the pair $(q,\tilde q)$ satisfies~\eqref{criticalhyp},
set $p:=\dfrac{q}{q-1}$ and choose any $r$ with $p<r<\tilde q$.
Fix $h\in C^{1}(\overline{\Omega})$, $\mu\in\mathbb{R}$ and $\lambda>0$.
    	\begin{enumerate}[label=\textup{(\roman*)}]
\item \emph{Non-resonant case.}
      Suppose $\lambda$ is \emph{not} one of the eigenvalues listed in
      Theorem~\ref{theoeigensequence}.
      If $N\ge 6$ and
      \[
        \frac{N+\sqrt{2N}}{N-2}\;\le q\;<\;\frac{N}{2},
      \]
      then there exists $\mu_{0}>0$ such that, whenever
      \[
        |\mu|+\|h\|_{(\tilde q)'}<\mu_{0},
      \]
      problem~\eqref{mainfinal} possesses two distinct, classical,
      non-trivial solutions.
      In the borderline case $q=N/2$ the same conclusion holds whenever
      $\lambda\ge\lambda_{0}$ for some $\lambda_{0}>0$.

\item \emph{Resonant case.}
      Assume $\lambda=\lambda_{k}$ for some eigenvalue $\lambda_{k}$
      of Theorem~\ref{theoeigensequence}.
      If $N\ge 8$ and
      \[
        \frac{N}{\,N-\sqrt{2(N+2)}\,}
        \;<\; q \;<\;
        \frac{p_{0}(N)}{p_{0}(N)-1},
      \]
      where $p_{0}(N)$ is  given in \eqref{defofp0},
      then the same  condition on $(\mu,h)$ as above guarantees
      two classical non-trivial solutions to~\eqref{mainfinal}.
\end{enumerate}
\end{corollary}
\begin{proof}
   The proof follows exaclty the same of Corollary \ref{corollaryhom} using \eqref{mainfinal} and \eqref{mainbifinal} instead of \eqref{mainhomfinal} and \eqref{mainhombifinal}. Notice that here we also need  $h\in C^{1}(\overline{\Omega})$.
\end{proof}

\subsection{Some remarks}

Let us finish this section with some important observations.

\begin{remark}
  The upper bounds on \(p\) used in Theorems~\ref{newhom1}, \ref{newhom2}, \ref{theononhomogeneousnonressonant}  and \ref{theoresonant_nonhom} can be written as
  \[
      p\;\le\;\sqrt{\frac{N}{2}}
      \;\;\Longleftrightarrow\;\;
      N\;\ge\;2p^{2}
      \qquad\text{and}\qquad
      p\;<\;\frac{N}{\sqrt{2(N+2)}}
      \;\;\Longleftrightarrow\;\;
      N(N+2p^{2})\;>\;4p^{2}.
  \] These conditions play the same technical role and are analogous to the growth restrictions that appear in the corresponding results for the \(p\)-Laplacian and the fractional \(p\)-Laplacian in
  \cite[Theorems 2.2 and 2.7]{Perera4}. By contrast, the \emph{lower} bounds on \(p\) are more delicate: they come from the
  asymptotic behaviour of the Aubin--Talenti concentration profiles when \(p\) is small.
\end{remark}

\begin{remark}
In the non-resonant case of both Corollary \ref{corollaryhom} and ~\ref{cor:nonhomogeneous}, the
admissible range for $(q,N)$ coincides with the \emph{non-critical region
on the critical hyperbola} introduced in~\cite{dosSantos-Melo}.  
In that work the authors obtained existence results for the \emph{homogeneous}
problem on the same range.  Outside this strip they required different
assumptions on the nonlinear terms and exponents.
\end{remark}

	\section{Proof of the main theorems}
	
	The proofs of the existence of a nontrivial solution for Problem \eqref{mainhombifinal} and two solutions for Problem \eqref{mainbifinal}  are carried out in several steps. We divide the reasoning in subsections: the first one concerns the compactness property for the associated functional and its threshold (Palais-Smale condition). Then we proceed to check some of the geometric properties that we need in order to apply the abstract theorem \ref{thm:2.1}. The third subsection is devoted to the construction of the concentration functions that allow the \textit{minimax} level to stay below the threshold. Then, we finish the proof of the main theorems in the last subsection.
	
	\subsection{On the Palais-Smale condition}
	
	The result below gives conditions for the functional associated with either problem \eqref{mainhombifinal} or \eqref{mainbifinal} to satisfy the \( (PS)_c \) condition. 
	
	\begin{lemma}\label{lem:PSbiharmonic}
		Let \( 1 < p < N/2 \), \( p_{2}^{*} := \frac{Np}{N - 2p} \), and \( p < r < p_{2}^{*} \). For parameters \( \lambda > 0 \), \( \mu \in \mathbb{R} \) and
		\( h \in L^{(p_{2}^{*})'}(\Omega) \), consider the functional $E$ defined in \eqref{functionalE}. Then there exists \( \kappa = \kappa(N, p, r, \Omega) > 0 \) such that \( E \)
		satisfies the Palais–Smale condition \( (\mathrm{PS})_{c} \) at every level
		\[
		c < \frac{2}{N} S_{2,p}^{N/2p}
		- \kappa \left(
		|\mu|^{(p_{2}^{*}/r)'}
		+ \|h\|_{(p_{2}^{*})'}^{(p_{2}^{*})'}
		\right),
		\]
		where \( S_{2,p} \) is defined in \eqref{S2p} and \( (p_{2}^{*}/r)' := \frac{p_{2}^{*}}{p_{2}^{*} - r} \).
	\end{lemma}
	
	\begin{proof}
		We follow \cite[Lemma 3.2]{Perera4}, adapting it to the
		biharmonic setting.
		
		The boundedness of a \( PS \) sequence is proved with standard, well-known arguments: let \( (u_n) \) be such that \( E(u_n) = c + o(1) \) and \( E^\prime(u_n)v = o(\|v\|) \). Then, taking \( u_n \) as a test function in the second equation, multiplying it by \( 1/p \), and subtracting from the energy identity, one has
		\[
		\mu\left(\frac{1}{p} - \frac{1}{r}\right)\|u_n\|_r^r + \left(\frac{1}{p} - \frac{1}{p_2^*}\right)\|u_n\|_{p_2^*}^{p_2^*} = c + o(1) + o(\|u_n\|) + \frac{p - 1}{p} \int_\Omega h u_n\,dx.
		\]
		If \( \mu \leq 0 \), we simply move the term \( \mu\left(1/p - 1/r\right)\|u_n\|_r^r \) to the right-hand side of the equation above and carry it within the estimates. So, suppose \( \mu > 0 \). Then, since Sobolev embeddings give us:
		\[
		\|u_n\|_r^r + \|u_n\|_{p_2^*}^{p_2^*} \leq C_1 + C_2 \|u_n\|,
		\]
		we can now estimate, using Hölder's inequality,
		\[
		\|u_n\|_p^p \leq |\Omega|^{1 - \frac{p}{p_2^*}} \|u_n\|_{p_2^*}^p \leq |\Omega|^{1 - \frac{p}{p_2^*}} (C_1 + C_2 \|u_n\|)^{\frac{p}{p_2^*}}.
		\]
		These last two inequalities can be inserted into \( E(u_n) = c + o(1) \), and we obtain
		\[
		\frac{1}{p} \|u_n\|^p \leq C_3 + C_4 \|u_n\| + \frac{\lambda}{p} |\Omega|^{1 - \frac{p}{p_2^*}} (C_1 + C_2 \|u_n\|)^{\frac{p}{p_2^*}},
		\]
		which gives the boundedness of \( (u_n) \). So, given any \( (PS)_c \) sequence \( (u_n) \), we have that,
		up to a subsequence, \( u_n \rightharpoonup u \) in \( W \) and strongly in \( L^\gamma(\Omega) \) for all \( 1 \leq \gamma < p_2^* \), for some \( u \in W \). Then, since \( E^\prime(u_n)u_n = o(\|u_n\|) \), we have
		\begin{equation}\label{tobl1}
			\|u_n\|^p = \|u_n\|_{p_2^*}^{p_2^*} + \lambda \|u_n\|_p^p + \mu \|u_n\|_r^r + \int_\Omega h\,u_n\,dx + o(1).
		\end{equation}
		It is also known that the weak limits of \( (PS)_c \) sequences are always critical points of functionals such as \( E \) (see, for instance, \cite[Proposition 3.1]{dosSantos1}). That means
		\begin{equation}\label{tobl2}
			\|u\|^p = \|u\|_{p_2^*}^{p_2^*} + \lambda \|u\|_p^p + \mu \|u\|_r^r + \int_\Omega h\,u\,dx.
		\end{equation}
		
		Set \( \tilde u_n := u_n - u \). The Brézis–Lieb relation (\cite[Theorem 1]{Brezis-Lieb}) then yields
		\begin{align*}
			\|u_n\|^p &= \|u\|^p + \|\tilde u_n\|^p + o(1), \\
			\|u_n\|_{p_2^*}^{p_2^*} &= \|u\|_{p_2^*}^{p_2^*} + \|\tilde u_n\|_{p_2^*}^{p_2^*} + o(1).
		\end{align*}
		Subtracting \eqref{tobl2} from \eqref{tobl1}, using these relations and the Sobolev embedding, we have
		\begin{equation}\label{downbelow}
			\|\tilde u_n\|^p = \|\tilde u_n\|_{p_2^*}^{p_2^*} + o(1).
		\end{equation}
		Applying the Sobolev inequality, we get
		\begin{equation} \label{eq:threshold_step}
			\|\tilde u_n\|^p \le \frac{1}{S_{2,p}^{p_2^*/p}} \|\tilde u_n\|^{p_2^*} + o(1).
		\end{equation}

		Now we estimate the energy level. From the energy identity, the Brézis–Lieb decomposition, \eqref{tobl2} and \eqref{downbelow}, we write:
		\begin{align}\nonumber
			c &= E(u_n) + o(1) \\\nonumber
			&= \frac{1}{p} \|\tilde u_n\|^p - \frac{1}{p_{2}^{*}} \|\tilde u_n\|^{p_{2}^{*}}_{p_{2}^{*}} 
			+ \frac{1}{p} \|u\|^p - \frac{1}{p_{2}^{*}} \|u\|^{p_{2}^{*}}_{p_{2}^{*}} 
			- \frac{\mu}{r} \|u\|_r^r - \frac{\lambda}{p} \|u\|_p^p - \int_\Omega h u\,dx + o(1)\\\label{toremark1}&= \frac2N\|\tilde u_n\|^p+\frac{2}{N}\|u\|^{p_{2}^{*}}_{p_{2}^{*}}+\mu\left(\frac1p-\frac1r\right)\|u\|_r^r-\left(1-\frac1p\right)\int_\Omega h\,u\,dx+o(1)\\\nonumber
			&\geq \frac2N\|\tilde u_n\|^p+\left[\frac{2}{N}\|u\|^{p_{2}^{*}}_{p_{2}^{*}}-|\mu|\left(\frac1p-\frac1r\right)|\Omega|^{1-r/p_2^*}\|u\|_{p_2^*}^r-\frac{p-1}p\|h\|_{(p_{2}^{*})^\prime}\|u\|_{p_{2}^{*}}\right]+o(1).
		\end{align}
		Now, one can use Young's inequality to infer existence of a constant \(\kappa=\kappa(N,p,r,\Omega)\) such that 
		\begin{equation*}
			\left[\frac{2}{N}\|u\|^{p_{2}^{*}}_{p_{2}^{*}}-|\mu|\left(\frac1p-\frac1r\right)|\Omega|^{1-r/p_2^*}\|u\|_{p_2^*}^r-\frac{p-1}p\|h\|_{(p_{2}^{*})^\prime}\|u\|_{p_{2}^{*}}\right]\geq - {\kappa}\left(|\mu|^{(p_2^*/r)'} + \|h\|_{(p_2^*)'}^{(p_2^*)'}\right)
		\end{equation*}
		These last estimates imply that
		\[
		c \geq\frac2N\|\tilde u_n\|^p
		- {\kappa}\left(|\mu|^{(p_2^*/r)'} + \|h\|_{(p_2^*)'}^{(p_2^*)'} \right) + o(1).
		\]
		
		Now, assume the energy level satisfies the strict inequality:
		\[
		c < \frac{2}{N} S_{2,p}^{N/2p} - \kappa \left(|\mu|^{(p_2^*/r)'} + \|h\|_{(p_2^*)'}^{(p_2^*)'} \right).
		\]
		Then the above estimate yields a contradiction with \eqref{eq:threshold_step} unless \( \|\tilde u_n\| \to 0 \). Thus, \( u_n \to u \) strongly in \( W^{2,p}(\Omega)\cap W_0^{1,p}(\Omega) \), completing the proof.
	\end{proof}
	\begin{remark}
		Observe that the term \(|\mu|\) appearing in the threshold condition can be completely dropped when \(\mu \geq 0\). That is, it is enough to require
		\[
		c < \frac{2}{N} S_{2,p}^{N/2p} - \kappa \|h\|_{(p_2^*)'}^{(p_2^*)'}
		\]
		in this case. Indeed, if \(\mu \geq 0\), then the term involving \(\mu\) in \eqref{toremark1} is nonnegative and can be omitted from the upper estimates in the proof. Nevertheless, $\mu$ is required to be small in the abstract theorem \ref{thm:2.1}, so, we need only the result of Lemma \ref{lem:PSbiharmonic}.
	\end{remark}
	
	\subsection{On the geometric conditions for $E_0$}

	We consider the sphere
	\[
	\mathcal{M} := \left\{ u \in W : I_p(u) = 1 \right\}
	= \left\{ u \in W : \|u\| = p^{1/p} \right\}.
	\]
	(Recall the definition of $I_p(u)$ given in \eqref{eq:2.3}, which gives, in our case, \(I_p(u)=\|u\|^p/p\)). The radial projection on \( \mathcal{M} \) is given by
	\[
	\pi_{\mathcal{M}}(u) := \frac{p^{1/p} u}{\|u\|}, \qquad \text{for } u \in W \setminus \{0\}.
	\]
	
	Let \( \Psi(u) := p\left( \int_\Omega |u|^p\,dx \right)^{-1} \) for \( u \in \mathcal{M} \), and for any \( a > 0 \), define the sublevel and superlevel sets:
	\[
	\Psi^a := \left\{ u \in \mathcal{M} : \Psi(u) \le a \right\}, \qquad
	\Psi_a := \left\{ u \in \mathcal{M} : \Psi(u) \ge a \right\}.
	\]
	
	We begin by stating a result of independent interest, which we need in the proof of the main results.
	
	\begin{proposition}\label{adapt_perera_lemma}
		If \(\lambda_k < \lambda_{k+1}\), then the sublevel set \(\Psi^{\lambda_k}\) has a compact symmetric subset \(C_0\) which is bounded in \(L^\infty(\Omega)\cap C^{2}_{loc}(\Omega)\) and \(i(C_0)=k\). 
	\end{proposition}
	\begin{proof}

By the abstract result \cite[Theorem~1.3]{Perera3} there exists a
compact symmetric set  \(C\subset\Psi^{\lambda_k}\)
with cohomological index \(i(C)=k\).
At this point no regularity beyond \(W^{2,p}(\Omega)\cap W^{1,p}_{0}(\Omega)\)
is guaranteed. Therefore, we implement an interaction scheme based on ideas in \cite[Theorem 2.3]{deGio-Lance}.

\smallskip

For \(w\in L^{p}(\Omega)\) let \(T(w)\in W^{2,p}(\Omega)\cap W^{1,p}_{0}(\Omega)\)
be the unique weak solution of
\[
   \Delta^{2}_{p}T(w)=|w|^{p-2}w
   \quad\text{in }\Omega,\qquad
   T(w)=\Delta T(w)=0\text{ on }\partial\Omega .
\]
Write \(K:=T\!\!\restriction_{W^{2,p}\cap W^{1,p}_{0}}\).
Then,
\(K\colon W^{2,p}\cap W^{1,p}_{0}\to W^{2,p}\cap W^{1,p}_{0}\) is compact.
Moreover, iterating \(K\) improves integrability:
for every \(u\in W^{2,p}\cap W^{1,p}_{0}\)
there exists \(m=m(N,p)\) such that
\[  
   K^{m}(u)\in L^{\infty}(\Omega)\cap C^{2}_{\mathrm{loc}}(\Omega)
   \quad\text{and}\quad
   K^{m}(u)\not\equiv0.
\]

Compactness of \(C\) and continuity of \(K\) imply that the same
iterate \(m\) works for \emph{all} \(u\in C\); hence
\(
   K^{m}(C)\subset L^{\infty}(\Omega)\cap C^{2}_{\mathrm{loc}}(\Omega).
\)

Define
\[
   H\colon C \longrightarrow \mathcal M,
   \qquad
   H(u):=\pi_{\mathcal M}\bigl(K^{m}(u)\bigr)
   =\frac{K^{m}(u)}{I_{p}\bigl(K^{m}(u)\bigr)^{1/p}} .
\]
Because \(\pi_{\mathcal M}\circ K^m\) is odd and continuous,
\(H\) is an odd homeomorphism onto its image.
Set \(C_{0}:=H(C)\). This finished the proof. \end{proof}

	From now on, we assume \( \lambda_k \leq \lambda < \lambda_{k+1} \), \( k \geq 1 \). The case \( 0 < \lambda < \lambda_1 \) is easier to handle since we can take $C_0=\emptyset$ for the geometric properties of $E_0$ and we omit the details. 
	
	\medskip
	
	Let us suppose, without loss of generality, that \( 0 \in \Omega \). Consider \( \eta : [0,\infty) \rightarrow [0,1] \) to be a smooth function such that \( \eta \equiv 0 \) on \( [0, 3/4] \) and \( \eta \equiv 1 \) on \( [1, +\infty) \). Then, let
	\[
	u_\rho(x) = \eta\left(\frac{|x|}{\rho}\right) u(x).
	\]
	Define \( \rho_0 = \mathrm{dist}(0, \partial\Omega) \). For each \( 0 < \rho < \rho_0 / 2 \), define the set
	\begin{equation}\label{C_ro}
		C(\rho) := \left\{ \pi_{\mathcal{M}}(u_\rho) ;\, u \in C_0 \right\},  
	\end{equation}
	where \( C_0 \) is given in Proposition \ref{adapt_perera_lemma}. Notice that all functions in $C(\rho)$ vanish in the ball $B_{3\rho/4}(0)$. In essence, we punch small ``holes'' in the functions of \(C_{0}\).
	These holes create space to insert the concentration test functions, whose supports lie entirely inside the holes and are thus disjoint from
the supports of the functions in \(C(\rho)\).
	Keeping the two supports separated greatly simplifies the forthcoming
estimates. We must verify that this alteration of \(C_{0}\) preserves all the key properties required for the subsequent arguments.
	
	\begin{lemma}\label{lemmadocro}
		The set \( C(\rho) \) given in \eqref{C_ro} is compact, symmetric and \( C(\rho) \subset \Psi^{\lambda_k + c_1\rho^{N - 2p}} \) for some \( c_1 > 0 \) and for all \( 0 < \rho < \min\{1, \rho_0 / 2\} \). Moreover, \( i(C(\rho)) = k \) whenever \( \lambda_k + c_1\rho^{N - 2p} < \lambda_{k+1} \).
	\end{lemma}
	
	\begin{proof}
		Since \( C_0 \subset C^2_{\mathrm{loc}}(\Omega) \) and is bounded in this space, there exists \( K > 0 \) such that
		\[
		\|u\|_{L^\infty(B_{\rho}(0))},\quad \|\nabla u\|_{L^\infty(B_{\rho}(0))},\quad \|\Delta u\|_{L^\infty(B_{\rho}(0))} \leq K \quad \text{for all } u \in C_0.
		\]
		Then, fixing \( u \in C_0 \), we have
		\begin{align*}
			\int_\Omega |\Delta u_\rho|^p\,dx 
			&= \int_{\Omega \setminus B_\rho(0)} |\Delta u|^p\,dx \\
			&\quad + \int_{B_\rho(0)} \left| 
			u \left( \frac{\eta^{\prime\prime}}{\rho^2} + \frac{N - 1}{|x|} \frac{\eta^\prime}{\rho} \right)
			+ 2 \nabla \eta \cdot \nabla u
			+ \eta \Delta u 
			\right|^p dx \\
			&\leq \int_{\Omega} |\Delta u|^p\,dx + K_1 \rho^N + K_2 \rho^{N - p} + K_3 \rho^{N - 2p} \leq p + K_4 \rho^{N - 2p},
		\end{align*}
		for some constants \( K_1, K_2, \dots \).
		
		Now,
		
		\[
		\int_{\Omega} |u_\rho|^p \, dx 
		\geq \int_{\Omega \setminus B_\rho(0)} |u|^p \, dx 
		= \int_{\Omega} |u|^p \, dx - \int_{B_\rho(0)} |u_\rho|^p \, dx 
		\geq \frac{p}{\lambda_k} - c \rho^N
		\]
		for some constant \( c>0 \). So
		\[
		\Psi(\pi_{\mathcal{M}}(u_\rho)) = 
		\frac{ \displaystyle \int_\Omega |\Delta u_\rho|^p \, dx }
		{ \displaystyle \int_\Omega |u_\rho|^p \, dx }
		\leq \lambda_k + c_1 \rho^{N-2p}
		\]
		for some constant \( c_1 > 0 \). Then 
		\( C(\rho) \subset \Psi^{\lambda_k + c_1 \rho^{N-2p}} \).  Notice that
		\( C(\rho) \) is compact and symmetric, since \( C_0 \) is a compact symmetric set and 
		\( u \mapsto \pi_{\mathcal{M}}(u_\rho) \) is an odd continuous map of \( C_0 \) onto \( C(\rho) \). Moreover,
		\[
		i(C(\rho)) \geq i(C_0) = k
		\]
		by the monotonicity of the index. If 
		\( \lambda_k + c_1 \rho^{N-2p} < \lambda_{k+1} \),
		then \( C(\rho) \subset \mathcal{M} \setminus \Psi_{\lambda_{k+1}} \), which implies that
		\[
		i(C(\rho)) \leq i(\mathcal{M} \setminus \Psi_{\lambda_{k+1}}) = k
		\]
		by \eqref{eq:index}, so \( i(C(\rho)) = k \) and we finish the proof. 
	\end{proof}
	
	Now, define \(E_0:W\rightarrow\mathbb R\) given by
	\begin{equation*}
		E_0(u)=\frac{1}{p}\int_{\Omega}|\Delta u|^pdx-\frac{\lambda}{p}\int_{\Omega}|u|^pdx-\frac{1}{p_2^*}\int_{\Omega}|u|^{p_2^*}dx.
	\end{equation*}
	
	This is the homogeneous functional that we use when applying Theorem \ref{thm:2.1}. We need to verify the geometric conditions in \eqref{eq:2.8} and.\eqref{eq:2.9}. As mentioned above, the advantage of working with \( C = C(\rho) \) instead of \( C_0 \) is that the functions in \( C \) have support disjoint from any function supported in \( B_{\rho/2}(0) \). This property simplifies the verification of the these  conditions. The next lemma deals with \eqref{eq:2.8}.
	
	\begin{lemma}\label{lemmageometry1}
		Let \(0<\rho\le \rho_0/2\). Then, for all \(w\in\mathcal M\backslash C(\rho)\) with support in \(\overline{B_{\rho/2}(0)}\), there exists \(R_0=R_0(\rho_0,w)\) such that \[
		\sup_{u\in A}E_0(Ru)\le0,
		\] 
		for all \(R\geq R_0\), where \(A:=\{\pi_{\mathcal M}((1-\tau)v+\tau w);\,v\in C(\rho),0\le\tau\le1\}\).
	\end{lemma}
	\begin{proof}
		Let \( u = \pi_{\mathcal{M}}((1 - \tau)v + \tau w) \in A \). For \( R > 0 \), consider the estimate
		\[
		E_0(Ru) \leq \int_\Omega \left( \frac{R^p}{p} |\Delta u|^p - \frac{R^{p_2^*}}{p_2^*} |u|^{p_2^*} \right) dx = R^p - \frac{R^{p_2^*}}{p_2^*} \|u\|_{p_2^*}^{p_2^*}.
		\]
		Therefore, to ensure that \( E_0(Ru) \le 0 \), it suffices to show that \( \|u\|_{p_2^*} \) is bounded away from zero on \( A \). By Hölder's inequality, it suffices to prove that \( \|u\|_p \) is bounded away from zero.
		
		Since \(v, w \in \mathcal{M}\), we have
		\[
		\|u\|_p^p 
		= \frac{ p\|(1 - \tau)v + \tau w\|_p^p }{ \|(1 - \tau)v + \tau w\|^p }.
		\]
		Then, using that \( v, w  \) have disjoint supports, we compute:
		\[
		\|(1 - \tau)v + \tau w\|_p^p = (1 - \tau)^p \|v\|_p^p + \tau^p \|w\|_p^p,
		\quad \text{and} \quad 
		\|(1 - \tau)v + \tau w\|^p = p(1 - \tau)^p + p\tau^p,
		\]
		so we conclude
		\[
		\|u\|_p^p = \frac{(1 - \tau)^p \|v\|_p^p + \tau^p \|w\|_p^p}{(1 - \tau)^p + \tau^p} 
		\ge \min\{ \|v\|_p^p, \|w\|_p^p \}.
		\]
		
		Thus, it suffices to show that \( \|v\|_p \) is bounded away from zero on \( C(\rho) \). But since \( C(\rho) \subset \Psi^{\lambda_k + c_1 \rho^{N - 2p}} \), by Lemma \ref{lemmadocro} we know
		\[
		\|v\|_p^p = \frac{p}{\Psi(v)} \ge \frac{p}{\lambda_k + c_1 \rho^{N - 2p}}\ge  \frac{p}{\lambda_k + c_1 (2\rho_0)^{N - 2p}},
		\]
		as required.
	\end{proof}
	
	Now choose \( 0 < \rho < \min\{1, \rho_0 / 2\} \) such that
	\begin{equation}\label{choiceofdeltaandrho}
		\lambda_k + c_1 \rho^{N - 2p} < \begin{cases}
			\lambda, & \text{if } \displaystyle 
			\lambda_{k}<\lambda<\lambda_{k+1},\\[6pt]
			\lambda_{k+1}, & \text{if } \displaystyle \lambda=\lambda_k,
		\end{cases}
	\end{equation}
	where $c_1>0$ is given in Lemma \ref{lemmadocro}. 
	We shall construct an element \( w_\rho \in \mathcal{M} \setminus C(\rho) \) such that, defining the cone-like joining set
	\begin{equation}\label{defofAdelta}
		A(\rho) := \left\{ 
		\pi_{\mathcal{M}} \left( (1 - \tau) v + \tau w_\rho \right) 
		: v \in C(\rho),\; 0 \le \tau \le 1 
		\right\},
	\end{equation}
	we obtain the energy estimate
	\begin{equation}\label{supinicial}
		\sup_{u \in A(\rho),\; t \ge 0} E_0(tu) < \frac{2}{N} S_{2,p}^{N/2p}.
	\end{equation}
	
	This will be achieved under the assumptions of Theorems \ref{newhom1} (or \ref{theononhomogeneousnonressonant}) and \ref{newhom2} (or \ref{theoresonant_nonhom}), for sufficiently small \( \rho \). The function \( w_\rho \) will be a suitable modification of the Aubin–Talenti concentration profiles. This construction will be carried out in the next subsection. Notice that, due to the definition of $A(\rho)$, \eqref{supinicial} will hold if and only if
	\begin{equation}\label{supwithoutpiM}
		\sup_{v\in C(\rho);\,t \ge 0;\,\tau\ge 0} E_0(tv+\tau w_\rho) < \frac{2}{N} S_{2,p}^{N/2p}.
	\end{equation}
	Moreover, in the next section we will see that $w_\rho$ will be chosen such that its support is disjoint from all functions on $C(\rho)$. Therefore, we can estimate $E_0(tv)$, for $v\in C(\rho)$ and $E_0(\tau w_\rho)$ separately. So, let us prove the following.
	
	\begin{lemma}\label{lemmasupofcdelta}
		For any $\rho$ satisfying \eqref{choiceofdeltaandrho}, we have
		\[
		\sup_{v\in C(\rho),\;t\ge 0} E_0(tv)\;\le\;
		\begin{cases}
			0, & \text{if } \displaystyle 
			\lambda_k< \lambda<\lambda_{k+1},\\[6pt]
			c\,\rho^{\,N(N-2p)/2p}, & \text{if } \displaystyle \lambda=\lambda_k,
		\end{cases}
		\]
		where the set $C(\rho)$ is given in Lemma \ref{lemmadocro}  and \(c>0\) is a constant.
	\end{lemma}
	
	\begin{proof}
		For any \(v\in C(\rho)\) and \(t\ge 0\),
		\[
		E_0(tv)=\frac{t^{p}}{p}\int_{\Omega}\bigl(|\Delta v|^{p}-\lambda|v|^{p}\bigr)\,dx
		-\frac{t^{p^{*}_2}}{p^{*}_2}\int_{\Omega}|v|^{p^{*}_2}\,dx,
		\]
		and using  \(C(\rho)\subset\Psi^{\lambda_k+c_1\rho^{\,N-2p}}\) (due to Lemma \ref{lemmadocro} and the choice of $\rho$ in \eqref{choiceofdeltaandrho}), we have
		\begin{equation}\label{eq:3.27}
			\frac1p\int_{\Omega}\bigl(|\nabla v|^{p}-\lambda|v|^{p}\bigr)\,dx
			=1-\frac{\lambda}{\Psi(v)}
			\;\le\;
			1-\frac{\lambda}{\lambda_k+c_1\rho^{\,N-2p}}.
		\end{equation}
		
		Thus, if $\lambda$ is not an eigenvalue, \eqref{choiceofdeltaandrho} gives \(E_0(tv)\le 0\).
		
		\medskip
		If \(\lambda=\lambda_k\) then, by \eqref{eq:3.27},
		\[
		\frac1p\int_{\Omega}\bigl(|\nabla v|^{p}-\lambda|v|^{p}\bigr)\,dx
		\;\le\;
		\frac{c_1\rho^{\,N-2p}}{\lambda_k+c_1\rho^{\,N-2p}}
		\;\le\;
		c_2\rho^{\,N-2p},
		\]
		where \(c_2=c_1/\lambda_k>0\).  Moreover,
		\[
		\frac1{p^{*}}\int_{\Omega}|v|^{p^{*}}\,dx\;\ge\;c_3>0,
		\]
		for some constant \(c_3>0\) (see the proof of Lemma \ref{lemmageometry1}).  Hence
		\[
		E_0(tv)\;\le\;c_2\rho^{\,N-2p}t^{p}-c_3\,t^{p^{*}_2},
		\]
		and maximization of the right-hand side over \(t\ge 0\) yield the stated upper bound.  This completes the proof. 
	\end{proof}
	
	\subsection{Construction of the concentration profiles}. 
	
	Let \(\phi\in\mathcal D^{2,p}(\mathbb R^N)\) be an extremal function that realizes \(S_{2,p}\), normalized in \(L^{p_2^*}(\mathbb R^N)\), radially symmetric, positive and such that $\phi(0)=1$. Lions in \cite[Corollary I.1]{Lions} proved that all of the positive solutions that realizes \(S_{2,p}\) are given by
	\begin{equation}\label{extremalfunctions}
		\phi_{\varepsilon,a}(x)=\varepsilon^{-\frac{N-2p}{p}}\phi\left(\frac{x-a}{\varepsilon}\right),\quad x\in\mathbb R^N,\quad a\in\mathbb R^N\quad\mbox{and}\quad\varepsilon>0.
	\end{equation}
	Notice that $S_{2,p}^{\frac{N-2p}{2p^2}}\phi_{\varepsilon,a}$ are exactly the regular positive solutions of $\Delta^2_pu=u^{p_2^*-1}$ in $\mathbb R^N$. 
	
	\medskip
	
	Moreover, to estimate the Laplacian of these extremal functions, we recall the theory developed in \cite{Hulshoff-Vander2, Lions}, which shows that \( -\Delta \phi = \widetilde{\phi}^{\frac{p}{p - 1} - 1} \) and \( -\Delta \widetilde{\phi} = \phi^{p_2^* - 1} \), where the function \( \widetilde{\phi} \) is the extremal function that realizes \( S_{2, \widetilde{p}} \), with \( \widetilde{p} = \frac{p_2^*}{p_2^* - 1} \). This follows from the identity
	\[
	\frac{1}{p/(p - 1)} + \frac{1}{p_2^*} = 1 - \frac{2}{N},
	\]
	which implies that \( \frac{p}{p - 1} = \widetilde{p}_2^{\,*} \). Therefore, \( \widetilde{\phi} \) can be viewed as the symmetric counterpart of \( \phi \), and the same structure described in \eqref{extremalfunctions} applies to \( \widetilde{\phi} \).
	
	\medskip
	Define now the following exponents:
	\begin{equation}\label{defdeA(N)eB(N)}
		A(N)=\frac{2N-2}{N}\quad\mbox{and}\quad B(N)=\frac{2N(N-1)}{N(N+2)-4}.
	\end{equation}
	These seemingly odd exponents are actually tied to the Serrin threshold \(N/(N-2)\): we have \(p=A(N)\) precisely when \(\frac{p}{p-1}-1=N/(N-2)\), and \(p=B(N)\) precisely when \(p_2^{*}-1=N/(N-2)\). Note also that $1<B(N)<A(N)<\sqrt{N/2}$ for all $N\geq 6$.
	Hulshof and van der Vorst \cite[Theorem 2]{Hulshoff-Vander2} proved that $\phi$ and $\widetilde\phi$ have the following asymptotic limits: there exist constants $a_1,a_2>0$ and $b_1,b_2>0$ such that
	\begin{align}\nonumber\lim_{r\rightarrow\infty}r^{\frac{N-2p}{p-1}}\phi(r)=a_1\quad\mbox{and}&\quad\lim_{r\rightarrow\infty}r^{N-2}\widetilde\phi(r)=b_2,&\mbox{if}&\quad A(N)<p\le\sqrt{\frac{N}{2}};\\\nonumber
		\label{asymptphi}\lim_{r\rightarrow\infty}\frac{r^{N-2}}{\log r}\phi(r)=a_1\quad\mbox{and}&\quad\lim_{r\rightarrow\infty}{r^{N-2}}\widetilde\phi(r)=b_2,&\mbox{if}&\quad p=A(N);
		\\\lim_{r\rightarrow\infty}r^{N-2}\phi(r)=b_1\quad\mbox{and}&\quad\lim_{r\rightarrow\infty}{r^{N-2}}\widetilde\phi(r)=b_2,&\mbox{if}&\quad B(N)<p<A(N);
		\\\nonumber\lim_{r\rightarrow\infty}r^{N-2}\phi(r)=b_1\quad\mbox{and}&\quad\lim_{r\rightarrow\infty}\frac{r^{N-2}}{\log r}\widetilde\phi(r)=a_2,&\mbox{if}&\quad p=B(N);
		\\\nonumber\lim_{r\rightarrow\infty}r^{N-2}\phi(r)=b_1\quad\mbox{and}&\quad\lim_{r\rightarrow\infty}r^{p_2^*(N-2)-N}\widetilde\phi(r)=a_2,&\mbox{if}&\quad 1<p<B(N),
	\end{align}
	where \(r=|x|\). 
	
	\medskip
	
	We will also need the asymptotic behavior of $\phi^\prime(r)$. In \cite[Equations (3.22) and (3.23)]{Hulshoff-Vander2} it is proven that
	\begin{equation}\label{asymptderivextremal}
		\lim_{r\rightarrow\infty}\frac{r\phi^\prime(r)}{\phi(r)}=\begin{cases}
			\displaystyle 2-\frac{N-2}{p-1}&\displaystyle\mbox{if}\quad A(N)<p\le\sqrt{\dfrac{N}{2}};\\[1.8ex]\displaystyle2-N&\displaystyle\mbox{if}\quad 1<p\leq A(N),
		\end{cases}
	\end{equation}

	In particular, notice that \(\phi\) behaves as \(|x|^{-\frac{N-2}{p-1}+2}\) when \(p>A(N)\) and as \(|x|^{-(N-2)}\) when \(p\leq A(N)\) as $|x|\rightarrow\infty$. So, in order to deal with problems such as \eqref{mainhombifinal} or \eqref{mainbifinal}, it is natural to consider 
	\[
	\frac{N-2p}{p-1}p\geq N,\quad\mbox{if}\quad p>A(N)\quad\mbox{or}\quad(N-2)p\geq N,\quad\mbox{if}\quad p\le A(N).
	\]
	Both situations require \(N\ge 6\), which can be easily checked. Also, the first inequality is equivalent to asking \(N\geq 2p^2\). 
	
	Now, consider \(\psi_\varepsilon=S_{2,p}^{\frac{N-2p}{2p^2}}\phi_{\varepsilon,0}\), where \(\phi_{\varepsilon,a}\) is given in \eqref{extremalfunctions}. Then, \(\int_{\mathbb R^N}|\Delta\psi_\varepsilon|^p\,dx=S_{2,p}\) and \(\int_{\mathbb R^N}\psi_\varepsilon^{p_2^*}\,dx=1\). Let \( \zeta : [0, \infty) \to [0, 1] \) be a smooth function such that \( \zeta(t) = 1 \) for \( t \leq 1/4 \) and \( \zeta(t) = 0 \) for \( t \geq 1/2 \), and let
	\begin{equation}\label{defofpsiandw}
	\psi_{\varepsilon, \rho}(x) = \zeta\left( \frac{|x|}{\rho} \right) \psi_\varepsilon(x), \qquad
	w_{\varepsilon, \rho}(x) = \frac{ \psi_{\varepsilon, \rho}(x) }{ \left( \displaystyle\int_{\mathbb{R}^N} \psi_{\varepsilon, \rho}^{p_2^*} \, dx \right)^{1/p_2^*} },
	\qquad 0 < \rho \leq \rho_0/2.
	\end{equation}
	Then \(\int_{\mathbb{R}^N} w_{\varepsilon, \rho}^{p_2^*} \, dx = 1\). Recalling the asymptotic behavior of $\phi$, we prove the following estimates:
	\begin{lemma}\label{estimatesfortalentifunctions}
		Suppose \(N\geq 6\) and  \(N\geq 2p^2\). Suppose $\varepsilon\leq\rho$ and  $\varepsilon$ small enough such that $|\log{\rho^{-1}\varepsilon}|\geq 1$. Then, there exist constants $C_1,C_2,C_3>0$ such that
		\[
		\left\| \Delta w_{\varepsilon, \rho} \right\|_{p}^{p} \leq
		\begin{cases}
			S_{2,p} + C_1\left(\rho^{-1}\varepsilon\right)^{\frac{N - 2p}{p - 1}}\,
			& \text{if}\quad\frac{2N - 2}{N}<p\le\sqrt{\frac{N}{2}} , \\[1.5ex]
			S_{2,p} + C_2|\log( \rho^{-1}\varepsilon)|^{\,p} \left(\rho^{-1}\varepsilon \right)^{N(p - 1)},
			& \text{if}\quad p = \frac{2N - 2}{N}, \\[1.5ex]
			S_{2,p} + C_3\left(\rho^{-1}\varepsilon \right)^{N(p - 1)},
			& \text{if}\quad 1<p < \frac{2N - 2}{N}.
		\end{cases}
		\]
	\end{lemma}
	\begin{proof} The proof of this lemma is based on ideas from \cite[Lemma 6.2]{dosSantos1}.
		We begin by estimating \(\|\psi_{\varepsilon,\rho}\|^{p_2^*}_{p_2^*}\). We have
		\begin{align*}
			\int_{\mathbb R^N}\psi^{p_2^*}_{\varepsilon,\rho}\,dx=&\int_{\mathbb R^N}\psi^{p_2^*}_{\varepsilon}\,dx+\int_{\mathbb R^N}\left(\left(\zeta\left(\frac{|x|}{\rho}\right)\right)^{p_2^*}-1\right)\psi^{p_2^*}_{\varepsilon}\,dx\\
			=&1+\rho^N\varepsilon^{-N}S_{2,p}^{N/2p}\int_{\mathbb R^N\backslash B_{1/4}(0)}\left(\zeta^{p_2^*}(|x|)-1\right)\phi^{\,p_2^*}\left(\frac{\rho|x|}{\varepsilon}\right)\,dx.
		\end{align*}
		This last integral is estimated using the asymptotic behavior of $\phi$ given in \eqref{asymptphi}. We separate the cases of $p$ accordingly and reach
		\begin{equation}\label{estimatep2*}
			\int_{\mathbb R^N}\psi^{p_2^*}_{\varepsilon,\rho}\,dx=
			\begin{cases}
				1 + \rho^{-\frac{N}{p-1}}O\left( \varepsilon^{\frac{N}{p-1}} \right),
				& \text{if}\quad A(N)<p\le\sqrt{\frac{N}{2}} , \\[1.5ex]
				1 + |\log(\rho^{-1}\varepsilon)|^{p_2^*} \rho^{-\frac{N}{p-1}}O\left( \varepsilon^{\frac{N}{p-1}} \right),
				& \text{if}\quad p = A(N), \\[1.5ex]
				1 + \rho^{-\left(p_2^*(N-2)-N\right)}O\left( \varepsilon^{p_2^*(N-2)-N} \right),
				& \text{if}\quad 1<p < A(N),
			\end{cases}
		\end{equation}
		where $A(N)$ is given in \eqref{defdeA(N)eB(N)}.
		To estimate the logarithmic factor, we used
		\[
		\bigl\lvert\log(\rho |x|/\varepsilon)\bigr\rvert
		=\bigl\lvert\log|x|-\log(\varepsilon/\rho)\bigr\rvert
		\le \bigl\lvert\log|x|\bigr\rvert
		+\bigl\lvert\log(\rho/\varepsilon)\bigr\rvert
		\le \bigl\lvert\log(\rho/\varepsilon)\bigr\rvert
		\bigl(1+\lvert\log|x|\rvert\bigr),
		\]
		which is valid whenever \(\lvert\log(\rho/\varepsilon)\rvert\ge 1\).  
		
		\medskip
		
		Now, we focus on estimating the Laplacian of this function. We shall use the following identity
		\begin{align}\nonumber
			\Delta \psi_{\varepsilon, \rho}(x) =&\, \Delta\left[\zeta\left( \frac{x}{\rho} \right) \psi_\varepsilon(x)\right]\\\nonumber=\,&\zeta\left( \frac{x}{\rho}\right)\,  
			\Delta \psi_\varepsilon\left(x\right)
			+  S_{2,p}^{\frac{N-2p}{2p^2}}\left(
			2 \rho^{-1}\varepsilon^{-\frac{N}{p_2^*}-1}\nabla \zeta\left( \frac{x}{\rho}\right)\,  \nabla \varphi\left( \frac{x}{\varepsilon} \right)
			+ \rho^{-2}\varepsilon^{-\frac{N}{p_2^*}}\varphi\left( \frac{x}{\varepsilon} \right) \Delta \zeta\left( \frac{x}{\rho}\right)
			\right)\\\label{beginshere}
			=&\,i_1+ S_{2,p}^{\frac{N-2p}{2p^2}}(i_2+i_3).
		\end{align}
		and estimate each of the integrals of these terms separetely. 
		
		The Laplacian of $\phi$ is estimated using its extremal counterpart $\widetilde\phi$, using the identity \( -\Delta \phi = \widetilde{\phi}^{\frac{1}{p - 1}} \). Starting with $i_1$:
		\begin{align*}
			\int_{\mathbb R^N}\left|\zeta\left( \frac{x}{\rho}\right)\,  
			\Delta \psi_\varepsilon\left(x  \right)\right|^p\,dx=&\int_{\mathbb R^N}\left|\Delta \psi_\varepsilon\left(x  \right)\right|^p\,dx+\int_{\mathbb R^N}\left(\left(\zeta\left(\frac{|x|}{\rho}\right)\right)^{p}-1\right)\left|\Delta \psi_\varepsilon\left(x  \right)\right|^p\,dx\\
			=&S_{2,p}+\rho^N\varepsilon^{-N}S_{2,p}^{(N-2p)/2p}\int_{\mathbb R^N\backslash B_{1/4}(0)}\left(\zeta^{p}(|x|)-1\right)\left|\Delta\phi\left(\frac{\rho|x|}{\varepsilon}\right)\right|^p\,dx\\
			=&S_{2,p}+\rho^N\varepsilon^{-N}S_{2,p}^{(N-2p)/2p}\int_{\mathbb R^N\backslash B_{1/4}(0)}\left(\zeta^{p}(|x|)-1\right)\left|\widetilde\phi\left(\frac{\rho|x|}{\varepsilon}\right)\right|^\frac{p}{p-1}\,dx.
		\end{align*}
		And now we apply the limits given in \eqref{asymptphi} for $\widetilde\phi$ to estimate this last integral. We get, then:
		\begin{equation}\label{estimatelaplacianp}
			\int_{\mathbb R^N}\left|\zeta\left( \frac{x}{\rho}\right)\,  
			\Delta \psi_\varepsilon\left(x  \right)\right|^p\,dx=
			\begin{cases}
				S_{2,p} + \rho^{-\frac{N - 2p}{p - 1}}O\left( \varepsilon^{\frac{N - 2p}{p - 1}} \right),
				& \text{if}\quad B(N)<p\le\sqrt{\frac{N}{2}} , \\[1.5ex]
				S_{2,p} + |\log(\rho^{-1}\varepsilon)|^{\frac{p}{p-1}} \rho^{-\frac{N - 2p}{p - 1}}O\left( \varepsilon^{\frac{N - 2p}{p - 1}} \right),
				& \text{if}\quad p = B(N), \\[1.5ex]
				S_{2,p} + \rho^{-N(p_2^*-1)}O\left( \varepsilon^{N(p_2^*-1)} \right),
				& \text{if}\quad 1<p < B(N),
			\end{cases}
		\end{equation}
		where $B(N)$ is given in \eqref{defdeA(N)eB(N)}. For later purpose, let us also highlight the part of these estimates. We also have proved that:
		\begin{equation}\label{estimatelaplacianpoutsideballs}
			\int_{\mathbb R^N\backslash B_{\rho/4}(0)}\left|\zeta\left( \frac{x}{\rho}\right)\,  
			\Delta \psi_\varepsilon\left(x  \right)\right|^p\,dx=
			\begin{cases}
				\rho^{-\frac{N - 2p}{p - 1}}O\left( \varepsilon^{\frac{N - 2p}{p - 1}} \right),
				& \text{if}\quad B(N)<p\le\sqrt{\frac{N}{2}} , \\[1.5ex]
				|\log(\rho^{-1}\varepsilon)|^{\frac{p}{p-1}} \rho^{-\frac{N - 2p}{p - 1}}O\left( \varepsilon^{\frac{N - 2p}{p - 1}} \right),
				& \text{if}\quad p = B(N), \\[1.5ex]
				\rho^{-N(p_2^*-1)}O\left( \varepsilon^{N(p_2^*-1)} \right),
				& \text{if}\quad 1<p < B(N).
			\end{cases}
		\end{equation}
		Next we estimate the integrals of \(i_{2}\) and \(i_{3}\).  In both cases the integration is taken outside a ball, because each term contains derivatives of the cut‐off function \(\zeta\) and we have \(\zeta\equiv 1\) in \(B_{1/4}(0)\). For \(i_2\), we have  
		\begin{align*}
			\int_{\mathbb R^N}\left|2 \rho^{-1}\varepsilon^{-\frac{N}{p_2^*}-1}\nabla \zeta\left( \frac{x}{\rho}\right)\,  \nabla \varphi\left( \frac{x}{\varepsilon} \right)\right|^p\,dx=2^p\rho^{N-p}\varepsilon^{-N+p}\int_{B_{1/2}(0)\backslash B_{1/4}(0)}\left|\nabla\zeta(x)\nabla\phi\left(\frac{\rho |x|}{\varepsilon}\right)\right|^pdx
		\end{align*}
		Now, observe that, due to \eqref{asymptderivextremal}, the gradient of $\phi$ satisfies $|\nabla\phi(x)|\leq C|\phi(x)|/|x|$ for all $x\neq 0$, for some constant $C>0$. So, we can go back to the estimates for $\phi$ in \eqref{asymptphi} and reach
		\begin{equation}\label{estimategradientp}
			\int_{\mathbb R^N}\left|2 \rho^{-1}\varepsilon^{-\frac{N}{p_2^*}-1}\nabla \zeta\left( \frac{x}{\rho}\right)\,  \nabla \varphi\left( \frac{x}{\varepsilon} \right)\right|^p\,dx=
			\begin{cases}
				\rho^{-\frac{N-2p}{p-1}}O\left( \varepsilon^{\frac{N-2p}{p-1}} \right),
				& \text{if}\quad A(N)<p\le\sqrt{\frac{N}{2}} , \\[1.5ex]
				|\log(\rho^{-1}\varepsilon)|^{p} \rho^{-\frac{N-2p}{p-1}}O\left( \varepsilon^{\frac{N-2p}{p-1}} \right),
				& \text{if}\quad p = A(N), \\[1.5ex]
				\rho^{-N(p-1)}O\left( \varepsilon^{N(p-1)} \right),
				& \text{if}\quad 1<p < A(N).
			\end{cases}
		\end{equation}
		The estimate for \(\int_{\mathbb{R}^{N}} |i_{3}|^{p}\,dx\) is obtained exactly in the same way. In \(i_{3}\) the powers of \(\rho\) and \(\varepsilon\) differ by one compared to those of \(i_{2}\); this difference is balanced by using the bound for \(\phi\) itself instead of the gradient estimate we needed for \(\nabla\phi\) before. Thus, we have
		\begin{equation}\label{estimatewithphip}
			\int_{\mathbb R^N}\left|\rho^{-2}\varepsilon^{-\frac{N}{p_2^*}}\varphi\left( \frac{x}{\varepsilon} \right) \Delta \zeta\left( \frac{x}{\rho}\right)\right|^p\,dx=
			\begin{cases}
				\rho^{-\frac{N-2p}{p-1}}O\left( \varepsilon^{\frac{N-2p}{p-1}} \right),
				& \text{if}\quad A(N)<p\le\sqrt{\frac{N}{2}} , \\[1.5ex]
				|\log(\rho^{-1}\varepsilon)|^{p} \rho^{-\frac{N-2p}{p-1}}O\left( \varepsilon^{\frac{N-2p}{p-1}} \right),
				& \text{if}\quad p = A(N), \\[1.5ex]
				\rho^{-N(p-1)}O\left( \varepsilon^{N(p-1)} \right),
				& \text{if}\quad 1<p < A(N).
			\end{cases}
		\end{equation}
		
		We will make use of the following numerical estimates found in \cite{Brezis-Nirenberg}: For all $\alpha,\beta\in\mathbb R$ and all $p>1$, there exists a constant $C>0$ such that
		\begin{equation}\label{numericalestimate}
			\left\{
			\begin{aligned}
				&\text{If } 1< p\le 3,\text{ then}\\[2pt]
				&\qquad 
				\Bigl|\;|\,\alpha+\beta|^{\,p}-|\alpha|^{\,p}-|\beta|^{\,p}
				-p\alpha\beta\bigl(|\alpha|^{\,p-2}+|\beta|^{\,p-2}\bigr)\;\Bigr|
				\\
				&\qquad\qquad
				\le 
				\begin{cases}
					C\,|\alpha|\,|\beta|^{\,p-1}, &\text{if }|\alpha|\ge|\beta|,\\[4pt]
					C\,|\alpha|^{\,p-1}|\beta|, &\text{if }|\alpha|\le|\beta|;
				\end{cases}
				\\[12pt]
				&\text{If } p\ge 3,\text{ then}\\[2pt]
				&\qquad 
				\Bigl|\;|\,\alpha+\beta|^{\,p}-|\alpha|^{\,p}-|\beta|^{\,p}
				-p\alpha\beta\bigl(|\alpha|^{\,p-2}+|\beta|^{\,p-2}\bigr)\;\Bigr|
				\\
				&\qquad\qquad
				\le 
				C\bigl(|\alpha|^{\,p-2}\beta^{2}+\alpha^{2}|\beta|^{\,p-2}\bigr).
			\end{aligned}
			\right.
		\end{equation}
		Considering $\alpha=i_1$ and $\beta=S_{2,p}^{\frac{N-2p}{2p^2}}(i_2+i_3)$, we still need to estimate mixed terms such as 
		\(\int |i_{1}|^{p-1} |i_{2}|\), \(\int |i_{1}|^{p-1} |i_{3}|\), \(\int |i_{1}|\,|i_{2}|^{p-1}\), and the remaining similar products in \eqref{numericalestimate}. We handle the mixed terms with H\"older’s inequality, combining it with the bounds already obtained in \eqref{estimatelaplacianpoutsideballs}, \eqref{estimategradientp} and \eqref{estimatewithphip}. The reasoning is identical for every combination. Observe that the bound for \(i_{1}\) involves the constant \(B(N)\), while the bounds for the other terms involve \(A(N)\). Considering the integral $\int|i_1|^{p-1}|i_2|$, notice that we can perform the following general scheme for all admissible $p$:
		\begin{align*}
			\int_{\mathbb R^N}|i_1|^{p-1}|i_2|\,dx\leq&\left(\int_{\mathbb R^N\backslash B_{\rho/4}(0)}\left|\zeta\left( \frac{x}{\rho}\right)\,  
			\Delta \psi_\varepsilon\left(x  \right)\right|^p\,dx\right)^{\frac{p-1}{p}}\left(\int_{\mathbb R^N}\left|2 \rho^{-1}\varepsilon^{-\frac{N}{p_2^*}-1}\nabla \zeta\left( \frac{x}{\rho}\right)\,  \nabla \varphi\left( \frac{x}{\varepsilon} \right)\right|^p\,dx\right)^\frac1p\\
			\leq&C\left(|\log(\rho^{-1}\varepsilon)|^{k_1}\rho^{-k_2}\varepsilon^{k_2}\right)^{\frac{p-1}{p}}\left(|\log(\rho^{-1}\varepsilon)|^{k_3}\rho^{-k_4}\varepsilon^{k_4}\right)^{\frac{1}{p}}\\
			=&C|\log(\rho^{-1}\varepsilon)|^{\frac{k_1(p-1)+k_3}{p}}\rho^{-\frac{k_2(p-1)+k_4}{p}}\varepsilon^{\frac{k_2(p-1)+k_4}{p}},
		\end{align*}
		where $k_1,k_3\ge 0$ and $k_2,k_4>0$ are constants depending on $p$, given in \eqref{estimatelaplacianpoutsideballs} and \eqref{estimategradientp}.
		When \(p>A(N)\) we have \(k_{1}=k_{3}=0\), \(k_{2}=\tfrac{N-2}{p-1}-2\) and \(k_{4}=\tfrac{N-2p}{p-1}\).  In this case, the mixed term is bounded by \(\int |i_{1}|^{p-1}|i_{2}| = \rho^{-\frac{N-2p}{p-1}}\,O\!\bigl(\varepsilon^{\frac{N-2p}{p-1}}\bigr)\).  
		Since the estimates for \(i_{2}\) and \(i_{3}\) coincide, the same bound applies to the integral involving \(i_{3}\).  Proceeding in the same way for the other ranges of \(p\) gives the full set of estimates listed below.
		\begin{equation}\label{estimatewithi_1p-1i_2ori_3}
			\int_{\mathbb R^N}|i_1|^{p-1}|i_2+i_3|\,dx=
			\begin{cases}
				\rho^{-\frac{N-2p}{p-1}}O\left( \varepsilon^{\frac{N-2p}{p-1}} \right),
				& \text{if}\quad A(N)<p\le\sqrt{\frac{N}{2}} , \\[1.5ex]
				|\log(\rho^{-1}\varepsilon)| \rho^{-\frac{N-2p}{p-1}}O\left( \varepsilon^{\frac{N-2p}{p-1}} \right),
				& \text{if}\quad p = A(N), \\[1.5ex]
				\rho^{-(N-2)}O\left( \varepsilon^{N-2} \right),
				& \text{if}\quad B(N)<p<A(N) , \\[1.5ex]
				|\log(\rho^{-1}\varepsilon)|\rho^{-(N-2)}O\left( \varepsilon^{N-2} \right),
				& \text{if}\quad p=B(N),\\[1.5ex]
				\rho^{-(p_2^*(N-2)-N)}O\left( \varepsilon^{p_2^*(N-2)-N} \right),
				& \text{if}\quad 1<p<B(N).
			\end{cases}
		\end{equation}
		Similarly, we reach the estimates for the integrals of $|i_1||i_2|^{p-1}$ and $|i_1||i_3|^{p-1}$. The general exponents to deal with now are 
		\begin{align*}
			\int_{\mathbb R^N}|i_1||i_2|^{p-1}dx+\!\int_{\mathbb R^N}|i_1||i_3|^{p-1}dx\leq C|\log(\rho^{-1}\varepsilon)|^{\frac{k_1+k_3(p-1)}{p}}\rho^{-\frac{k_2+k_4(p-1)}{p}}\varepsilon^{\frac{k_2+k_4(p-1)}{p}},
		\end{align*}
		which lead to
		\begin{equation}\label{estimatewithi_1i_2ori_3p-1}
			\int_{\mathbb R^N}\!\!\!|i_1||i_2+i_3|^{p-1}dx=
			\begin{cases}
				\rho^{-\frac{N-2p}{p-1}}O\left( \varepsilon^{\frac{N-2p}{p-1}} \right),
				& \text{if}\quad A(N)<p\le\sqrt{\frac{N}{2}} , \\[1.5ex]
				|\log(\rho^{-1}\varepsilon)|^{p-1} \rho^{-\frac{N-2p}{p-1}}O\left( \varepsilon^{\frac{N-2p}{p-1}} \right),
				& \text{if}\quad p = A(N), \\[1.5ex]
				\rho^{-\frac{1}p\left(\frac{N-2p}{p-1}+N(p-1)^2\right)}O\left( \varepsilon^{\frac{1}p\left(\frac{N-2p}{p-1}+N(p-1)^2\right)} \right),
				& \text{if}\quad B(N)<p<A(N) , \\[1.5ex]
				|\log(\rho^{-1}\varepsilon)|^{\frac{1}{p-1}}\rho^{-N(p-1)-2p_2^*}O\left( \varepsilon^{N(p-1)+2p_2^*} \right),
				& \text{if}\quad p=B(N),\\[1.5ex]
				\rho^{-N(p-1)-2p_2^*}O\left( \varepsilon^{N(p-1)+2p_2^*} \right),
				& \text{if}\quad 1<p<B(N).
			\end{cases}
		\end{equation}
		
		We still need to estimate the integrals of $|i_1|^{p-2}|i_2+i_3|^2$ and $|i_1|^{2}|i_2+i_3|^{p-2}$,
which arise only when \(p>3\).  Because \(A(N)<3\), the relevant range is now \(A(N)<p\le\sqrt{N/2}\).  In this case, we avoid Hölder inequality and instead bound the integrands directly. Recalling that \( -\Delta \phi = \widetilde{\phi}^{\frac{1}{p - 1}}\) and using \eqref{asymptphi} and \eqref{asymptderivextremal}, we get
		\begin{align*}
			\int_{\mathbb R^N} \!\!|i_1|^{p-2}|i_2+i_3|^2dx\leq&\,C\rho^N\varepsilon^{-\frac{N}{p}(p-2)}\!\!\!\int_{B_{\frac12}(0)\backslash B_{\frac14}(0)}\!\left| \widetilde\phi\left(\frac{\rho x}{\varepsilon}\right)\!\right|^{\frac{p-2}{p-1}}\!\left| \rho^{-1}\varepsilon^{\!-\frac{N}{p_2^*}-1}\nabla \varphi\left( \frac{\rho x}{\varepsilon} \right)
			+ \rho^{-2}\varepsilon^{\!-\frac{N}{p_2^*}}\varphi\left( \frac{\rho x}{\varepsilon} \right)\!\right|^2\!\!dx\\
			\leq&\,C\rho^{N+(2-N)\frac{p}{p-1}}\varepsilon^{-N+(N-2)\frac{p}{p-1}}\int_{B_{\frac12}(0)\backslash B_{\frac14}(0)}\!\!\!\!\!\!\left|x \right|^{(2-N)\frac{p-2}{p-1}}\left(|x|^{2-2\frac{N-2}{p-1}}\,   
			+ |x|^{4-2\frac{N-2}{p-1}}\right)dx\\
			\leq&\,C\rho^{-\frac{N-2p}{p-1}}\varepsilon^{\frac{N-2p}{p-1}}.
		\end{align*}
		We reach these same exponents when dealing with $|i_1|^{2}|i_2+i_3|^{p-2}$. Therefore, we may conclude
		\begin{equation}\label{lastestimatemixed}
			\int_{\mathbb R^N}|i_1|^{p-2}|i_2+i_3|^2dx+\int_{\mathbb R^N}|i_1|^{2}|i_2+i_3|^{p-2}dx=\rho^{-\frac{N-2p}{p-1}}O(\varepsilon^{\frac{N-2p}{p-1}}),\quad\mbox{if}\quad A(N)<p\le\sqrt{\frac N2}.
		\end{equation}
		We are ready to go back to \eqref{beginshere}. Using \eqref{numericalestimate}, we see that
		\begin{equation*}
			\begin{array}{l}
				\displaystyle\int_{\mathbb R^N}\left||\Delta \psi_{\varepsilon, \rho}(x)|^p-\left|\zeta\left( \frac{x}{\rho}\right)\,  
				\Delta \psi_\varepsilon\left(x  \right)\right|^p\right|dx=\int_{\mathbb R^N}\left||i_1+ S_{2,p}^{\frac{N-2p}{2p^2}}(i_2+i_3)|^p-\left|i_1\right|^p\right|dx\leq\\\leq\begin{cases}
					\displaystyle C\int_{\mathbb R^N}\!\!\!\!\left(|i_2|^p+|i_3|^p+|i_1|^{p-1}|i_2+i_3|+|i_1||i_2+i_3|^{p-1}\right)dx&\mbox{if }1<p\le 3,\medskip\\\displaystyle C\int_{\mathbb R^N}\!\!\!\!\left(|i_2|^p\!+\!|i_3|^p+|i_1|^{p-1}|i_2+i_3|\!+\!|i_1||i_2\!+\!i_3|^{p-1}\!+\!|i_1|^{p-2}|i_2\!+\!i_3|^2\!+\!|i_1|^2|i_2\!+\!i_3|^{p-2}\right)dx&\mbox{if }p>3.
				\end{cases}
			\end{array}
		\end{equation*}
		In every integral estimate—namely \eqref{estimatelaplacianp}, \eqref{estimategradientp}, \eqref{estimatewithphip}, \eqref{estimatewithi_1p-1i_2ori_3}, \eqref{estimatewithi_1i_2ori_3p-1} and \eqref{lastestimatemixed}—the factors involving \(\rho\) and \(\varepsilon\) appear in the form \(\rho^{-k}\varepsilon^{k}\).   The last step to estimate \(\int|\Delta \psi_{\varepsilon, \rho}|^p\) is to list the corresponding exponents \(k\) for every range of \(p\) and pick the smallest one in each case; that smallest exponent controls the overall estimate, since we are assuming $\varepsilon\rho^{-1}\leq 1$. So,
		\begin{itemize}
			\item Case \(A(N)<p\le\sqrt{\frac{N}{2}}\): the powers of $\rho^{-1}\varepsilon$ that appear in all the estimates are the same, namely $(N-2p)/(p-1)$.
			\item \textbf{Case \(\displaystyle p = A(N)\).}
			In this situation every power of \(\rho^{-1}\varepsilon\) is still
			\(\frac{N-2p}{p-1}\), but some terms gain extra factors
			\( \lvert\log(\rho^{-1}\varepsilon)\rvert \),
			\( \lvert\log(\rho^{-1}\varepsilon)\rvert^{\,p-1} \) or
			\( \lvert\log(\rho^{-1}\varepsilon)\rvert^{\,p} \).
			Because we are assuming \( \lvert\log(\rho^{-1}\varepsilon)\rvert \le 1 \),
			the factor with exponent \(p\) dominates,
			and therefore controls all the remaining logarithmic corrections.
			\item Case \(B(N)<p<A(N)\): Here, the powers are 
			\[
			\frac{N-2p}{p-1},\quad N(p-1),\quad N-2,\quad \frac{1}{p}\left(\frac{N-2p}{p-1}+N(p-1)^2\right).
			\]
			In this case $N(p-1)$ is the smallest. 
			\item  Case \(p=B(N)\): some powers of $\varepsilon$ are also multiplied by powers of $|log(\rho^{-1}\varepsilon)|$ but all exponents of them are greater than $N(p-1)$, which remains the smallest one.
			\item  Case \(1<p<B(N)\): Now the powers of $\rho^{-1}\varepsilon$ are:
			\[
			N(p_2^*-1),\quad N(p-1),\quad p_2^*(N-2)-N,\quad N(p-1)+2p_2^*,
			\]
			which still gives $N(p-1)$ as the correct one to choose. 
		\end{itemize}
		So, we have proved that: \
		
		\begin{equation}\label{estimatelaplacianpsiepsilonp}
			\int_{\mathbb R^N}\left|\Delta \psi_{\varepsilon, \rho}\right|^p\,dx\leq
			\begin{cases}
				S_{2,p}+C\left( \rho^{-1}\varepsilon \right)^{\frac{N-2p}{p-1}},
				& \text{if}\quad A(N)<p\le\sqrt{\frac{N}{2}} , \\[1.5ex]
				S_{2,p}+C|\log(\rho^{-1}\varepsilon)|^{p} \left( \rho^{-1}\varepsilon\right)^{\frac{N-2p}{p-1}} ,
				& \text{if}\quad p = A(N), \\[1.5ex]
				S_{2,p}+C\left( \rho^{-1}\varepsilon \right)^{N(p-1)},
				& \text{if}\quad 1<p < A(N).
			\end{cases}
		\end{equation}
		Finally, we estimate $\|\Delta w_{\varepsilon,\rho}\|^p_p$ by means of \eqref{estimatep2*} and \eqref{estimatelaplacianpsiepsilonp}.
	\end{proof}

	We also need to estimate the $L^p$ norms of $w_{\varepsilon, \rho}$. This is the subject of the next lemma.
	\begin{lemma}\label{lemmaLpnormestimate} Suppose \(N\geq 6\) and  \(N\geq 2p^2\). Suppose $4\varepsilon<\rho$. Then, there exist constants $C_1,C_2>0$ such that
		\begin{equation*}
			\left\| w_{\varepsilon, \rho} \right\|_{p}^{p} \geq
			\begin{cases}
				
				C_1 \varepsilon^{\,2p}|\log (\rho^{-1}\varepsilon)|,
				& \text{if}\quad  p=\sqrt{\frac{N}{2}}, \\[2ex]
				
				C_2 \varepsilon^{\,2p},
				& \text{if}\quad 1< p<\sqrt{\frac{N}{2}}.
			\end{cases}
		\end{equation*}
	\end{lemma}
	\begin{proof}
		The required bounds follow directly from the asymptotic profile of
		\(\phi\); in fact they are simpler than those in the previous lemma.
		We only need a lower estimate for the \(L^{p}\)-norm of
		\(\psi_{\varepsilon,\rho}\).  The upper estimate for the
		\(L^{p_{2}^{*}}\)-norm is already available in
		\eqref{estimatep2*}, and combining the two gives the desired result. 
		
		\medskip
		
		Since $\rho/4\varepsilon>1$, we have 
		\begin{align*}
			\int_{\mathbb R^N}|\psi_{\varepsilon,\rho}|^p\,dx\geq&\,C\varepsilon^{-(N-2p)}\int_{B_{\frac{\rho}{4}}(0)}\left(\phi\left(\frac{|x|}{\varepsilon}\right)\right)^pdx\\
			\geq&\,C\varepsilon^{2p}\left[\int_{B_{\frac{\rho}{4\varepsilon}}(0)\backslash B_1(0)}\phi(|x|)^pdx+|B_1(0)|\right]\\
			=&\,C\varepsilon^{2p}+C\varepsilon^{2p}\int_1^{\rho/4\varepsilon}\phi^p(r)r^{N-1}dr.
		\end{align*}
		From this point, we only need to separate the cases for $p$, considering the limits given in \eqref{asymptphi}, noticing that the case $N=2p^2$ gives the borderline $\phi^p(r)\geq Cr^{-N}$ for all $r\geq1$. All other cases give bounds for $\phi$ such that the last integral is positive and finite and can be absorbed into the generic constant \(C\). We omit the details.
	\end{proof}
	
	\subsection{Construction of $w_\rho$ and proof of the main theorems}
	
	Now, for any sufficiently small $\varepsilon$ and $\rho$ such that \eqref{choiceofdeltaandrho}, $4\varepsilon<\rho$ and $|\log(\rho^{-1}\varepsilon)|\ge 1$ hold, define
	\[
	w_\rho=\pi_{\mathcal M}(w_{\varepsilon,\rho}),
	\]
	where $w_{\varepsilon,\rho}$ is defined in \eqref{defofpsiandw} and satisfies the estimates in lemmas \ref{estimatesfortalentifunctions} and \ref{lemmaLpnormestimate}. We recall the definition of $A(\rho)$ given in \eqref{defofAdelta} and the  required equivalent estimates in \eqref{supinicial} and \eqref{supwithoutpiM}.  Notice that by the definition of $A(\rho)$, \eqref{supwithoutpiM} is also equivalent to
	\begin{equation}\label{supwithoutpiMfinal}
		\sup_{v\in C(\rho);\,t \ge 0;\,\tau\ge 0} E_0(tv+\tau w_{\varepsilon,\rho}) < \frac{2}{N} S_{2,p}^{N/2p}.
	\end{equation}
	As already mentioned, the definition of $C(\rho)$ in \eqref{C_ro} ensures that $w_\rho$ has disjoint support from any of its functions. Thus, the estimate \eqref{supwithoutpiMfinal} can be split. So, it remains to prove that   
		\begin{equation}\label{supwithoutpiMfinal2}
		\sup_{v\in C(\rho);\,t \ge 0} E_0(tv)+\sup_{\,\tau\ge 0} E_0(\tau w_{\varepsilon,\rho}) < \frac{2}{N} S_{2,p}^{N/2p}.
	\end{equation}
	Lemma \ref{lemmasupofcdelta} already gives what we need for the first supremum. We are ready to prove  the estimates concerning $w_{\varepsilon,\rho}$ in the next lemma.
    
	\begin{lemma}\label{lemmaestimatewepsilonrhofinal}
Suppose that $\varepsilon$ and $\rho$ are such that   \eqref{choiceofdeltaandrho}, $4\varepsilon<\rho$ and $|\log(\rho^{-1}\varepsilon)|\ge 1$ hold. We have
\[
\sup_{\tau\ge 0} E_{0}\!\bigl(\tau\,w_{\varepsilon,\rho}\bigr)\;\le\frac{2}{N}\;
		\begin{cases}
        \displaystyle\,\Bigl[S_{2,p} + C_1\left(\rho^{-1}\varepsilon\right)^{2p}-C_2 \lambda\varepsilon^{\,2p}|\log (\rho^{-1}\varepsilon)|\Bigr]^{N/2p}\,
			& \text{if}\quad p=\sqrt{\frac{N}{2}} , \\[1.5ex]
			\displaystyle\,
\Bigl[S_{2,p} + C_1\left(\rho^{-1}\varepsilon\right)^{\frac{N - 2p}{p - 1}}-C_3\lambda \varepsilon^{\,2p}\Bigr]^{N/2p}\,
			& \text{if}\quad\frac{2N - 2}{N}<p<\sqrt{\frac{N}{2}} , \\[1.5ex]
			\displaystyle\,
\Bigl[S_{2,p} + C_4|\log( \rho^{-1}\varepsilon)|^{\,p} \left(\rho^{-1}\varepsilon \right)^{\frac{N - 2p}{p - 1}}-C_3\lambda \varepsilon^{\,2p}\Bigr]^{N/2p},
			& \text{if}\quad p = \frac{2N - 2}{N}, \\[1.5ex]
			\displaystyle\,
\Bigl[S_{2,p} + C_5\left(\rho^{-1}\varepsilon \right)^{N(p - 1)}-C_3\lambda\varepsilon^{\,2p}\Bigr]^{N/2p},
			& \text{if}\quad 1<p < \frac{2N - 2}{N}.
		\end{cases}
    \]
\end{lemma}
	\begin{proof}
Notice that, since $\|w_{\varepsilon,\rho}\|_{p^{*}_2}=1$ (see \eqref{defofpsiandw}), we have
\[
E_{0}\bigl(\tau w_{\varepsilon,\rho}\bigr)
  =\frac{\tau^{p}}{p}\int_{\Omega}\!\bigl(|\Delta w_{\varepsilon,\rho}|^{p}
        -\lambda\,w_{\varepsilon,\rho}^{\,p}\bigr)\,dx
   \;-\;\frac{\tau^{p^{*}_2}}{p^{*}_2}.
\]
 maximizing the right–hand side with respect to
$\tau\ge 0$ yields
\begin{equation*}
\sup_{\tau\ge 0} E_{0}\bigl(\tau w_{\varepsilon,\rho}\bigr)\;=\;
\frac2N\Bigl[
  \int_{\Omega}\!\bigl(|\Delta w_{\varepsilon,\rho}|^{p}
        -\lambda\,w_{\varepsilon,\rho}^{\,p}\bigr)\,dx
\Bigr]^{N/2p}.
\end{equation*}
Observe that $2p=\frac{N - 2p}{p - 1}$ when $p=\sqrt{\frac N2}$ and $\frac{N - 2p}{p - 1}=N(p-1)$ when $p=\frac{2N-2}{N}$. The conclusion follows from Lemmas \ref{estimatesfortalentifunctions} and \ref{lemmaLpnormestimate}.
\end{proof}

We are now ready to prove Theorems \ref{newhom1}, \ref{newhom2}, \ref{theononhomogeneousnonressonant} and \ref{theoresonant_nonhom}.

\begin{proof}
\noindent\textbf{[Proof of Theorems \ref{newhom1} and \ref{theononhomogeneousnonressonant}]}. All that remains is to verify \eqref{supwithoutpiMfinal2}. This is the case where $\lambda$ is not an eigenvalue, and so we can fix $\rho>0$ sufficiently small such that \eqref{choiceofdeltaandrho} holds and let $\varepsilon$ be such that $4\varepsilon<\rho$ and $|\log(\rho^{-1}\varepsilon)|\ge 1$. Then, due to Lemmas \ref{lemmasupofcdelta}, one can see that \eqref{supwithoutpiMfinal2} follows from Lemma \ref{lemmaestimatewepsilonrhofinal} if we choose $\varepsilon$ sufficiently small such that:
\begin{align*}
C_1\left(\rho^{-1}\varepsilon\right)^{2p}-C_2\lambda \varepsilon^{\,2p}|\log (\rho^{-1}\varepsilon)|<0,&\quad\mbox{if}\quad p=\sqrt{\frac N2};\\
C_1\left(\rho^{-1}\varepsilon\right)^{\frac{N - 2p}{p - 1}}-C_3 \lambda\varepsilon^{\,2p}<0,&\quad\mbox{if}\quad \frac{2N-2}{N}<p<\sqrt{\frac N2};\\C_4|\log( \rho^{-1}\varepsilon)|^{\,p} \left(\rho^{-1}\varepsilon \right)^{\frac{N - 2p}{p - 1}}-C_3 \lambda\varepsilon^{\,2p}<0,&\quad\mbox{if}\quad p=\frac{2N-2}{N};\\ C_5\left(\rho^{-1}\varepsilon \right)^{N(p - 1)}-C_3 \lambda\varepsilon^{\,2p}<0,&\quad\mbox{if}\quad 1<p<\frac{2N-2}{N}.
\end{align*}
For any $\lambda>0$, this is always possible for the first three inequalities, but the last can occur only if $p>N/(N-2)$. In case $p=N/(N-2)$, we have $N(p-1)=2p$, and so, we can choose a sufficiently large $\lambda_0$ so $C_5\left(\rho^{-1}\varepsilon \right)^{N(p - 1)}-C_3 \lambda\varepsilon^{\,2p}<0$ for all $\lambda\ge\lambda_0$. Notice that since $N\geq 6$, we always have $N/(N-2)<(2N-2)/N$. This finishes the proof.
\end{proof}

The proof of Theorems \ref{newhom1} and \ref{theoresonant_nonhom} (the resonant case) requires extra care. Since the estimate for \(E_{0}(t v)\) with \(v\in C(\rho)\) (Lemma \ref{lemmasupofcdelta}) now enters together with the concentration estimate depending on \(\varepsilon\), we must choose \(\rho\) and \(\varepsilon\) in a coordinated way.

\begin{proof}
\noindent\textbf{[Proof of Theorems \ref{newhom2} and \ref{theoresonant_nonhom}]}. As in the proof of the non-resonant cases, all we need to do is to verify that \eqref{supwithoutpiMfinal2} holds. Notice that, by Lemmas \ref{lemmasupofcdelta} and \ref{lemmaestimatewepsilonrhofinal}, we have

\begin{align*}
&\sup_{v\in C(\rho);\,t \ge 0} E_0(tv)+\sup_{\,\tau\ge 0} E_0(\tau w_{\varepsilon,\rho})\leq\\& \le c\,\rho^{\,N(N-2p)/2p}+
\frac{2}{N}\;
		\begin{cases}
        \displaystyle\,\Bigl[S_{2,p} + C_1\left(\rho^{-1}\varepsilon\right)^{2p}-C_2\lambda \varepsilon^{\,2p}|\log (\rho^{-1}\varepsilon)|\Bigr]^{N/2p}\,
			& \text{if}\quad p=\sqrt{\frac{N}{2}} , \\[1.5ex]
			\displaystyle\,
\Bigl[S_{2,p} + C_1\left(\rho^{-1}\varepsilon\right)^{\frac{N - 2p}{p - 1}}-C_3\lambda\varepsilon^{\,2p}\Bigr]^{N/2p}\,
			& \text{if}\quad\frac{2N - 2}{N}<p<\sqrt{\frac{N}{2}} , \\[1.5ex]
			\displaystyle\,
\Bigl[S_{2,p} + C_4|\log( \rho^{-1}\varepsilon)|^{\,p} \left(\rho^{-1}\varepsilon \right)^{\frac{N - 2p}{p - 1}}-C_3\lambda \varepsilon^{\,2p}\Bigr]^{N/2p},
			& \text{if}\quad p = \frac{2N - 2}{N}, \\[1.5ex]
			\displaystyle\,
\Bigl[S_{2,p} + C_5\left(\rho^{-1}\varepsilon \right)^{N(p - 1)}-C_3\lambda\varepsilon^{\,2p}\Bigr]^{N/2p},
			& \text{if}\quad 1<p < \frac{2N - 2}{N}.
		\end{cases}
\end{align*}
From this point, let us consider $0<\alpha<1$ to be chosen later and let $\rho=\varepsilon^\alpha$. Suppose that $\varepsilon$ is sufficiently small so $\varepsilon^{1-\alpha}<1/4$ and $|\log(\varepsilon^{1-\alpha})|\ge 1$. The right-hand side of this last estimate depends on various locations of $p$, so let us see them case by case. 
\begin{itemize}
    \item Case $p=\sqrt\frac N2$: we have
    \[
    \sup_{v\in C(\rho);\,t \ge 0} E_0(tv)+\sup_{\,\tau\ge 0} E_0(\tau w_{\varepsilon,\rho}) \le c\,\varepsilon^{\,\alpha N(N-2p)/2p}+\frac{2}{N}\Bigl[S_{2,p} + C_1\varepsilon^{2p(1-\alpha)}-C_2 \lambda\varepsilon^{\,2p}|\log (\varepsilon^{1-\alpha})|\Bigr]^{N/2p}
    \]
    Notice that, to make the right–hand side of this estimate smaller than \( \tfrac{2}{N} S_{2,p}^{N/2p} \), we would have to choose \( 0 < \alpha < 1 \) so that \(2p(1-\alpha) > 2p,\) which is impossible. Hence, this case cannot occur and is ruled out.
    \item Case \(\frac{2N-2}{N}\leq p<\sqrt{\frac N2}\): Here, the difference between \(p=\frac{2N-2}{N}\) and \(p>\frac{2N-2}{N}\) is just a term with \(\log\) which does not interfere and can be ignored. Now, to make the right–hand side of this estimate smaller than \( \tfrac{2}{N} S_{2,p}^{N/2p} \) we need to choose $0<\alpha<1$ such that 
    \[
    \alpha N\frac{N-2p}{2p}>2p\quad\mbox{and}\quad(1-\alpha)\frac{N-2p}{p-1}>2p.
    \]
This is possible if and only if
\[
\frac{4p^2}{N(N-2p)}<1\quad\mbox{and}\quad4p^2<N(N-2p^2).
\]
The first inequality holds for all $N\geq 6$ and $p\leq\sqrt{\frac N2}$. But the second holds if and only if $p<\frac{N}{\sqrt{2(N+2)}}$. Thus, this choice of $\alpha$ is only possible if \[ \frac{2N-2}{N}<\frac{N}{\sqrt{2(N+2)}}. \] This inequality is true only when $N\ge 8$. Moreover, notice that we  always have \(\frac{N}{\sqrt{2(N+2)}}<\sqrt{\frac N2}\).
\item Case \(1<p<\frac{2N-2}{N}\): Now, we have
 \[
    \sup_{v\in C(\rho);\,t \ge 0} E_0(tv)+\sup_{\,\tau\ge 0} E_0(\tau w_{\varepsilon,\rho}) \le c\,\varepsilon^{\,\alpha N(N-2p)/2p}+\frac{2}{N}\Bigl[S_{2,p} + C_1\varepsilon^{N(p-1)(1-\alpha)}-C_2 \lambda\varepsilon^{\,2p}\Bigr]^{N/2p}.
    \]
    Therefore, we need to choose $0<\alpha<1$ such that
     \[
    \alpha N\frac{N-2p}{2p}>2p\quad\mbox{and}\quad(1-\alpha)N(p-1)>2p.
    \]
    The second inequality already gives us the restriction $p>N/(N-2)$, as in the proof of the non-resonant case.  This is not enough, though, as we see that the choice of $\alpha$ is possible if and only if
    \[
\frac{4p^2}{N(N-2p)}<1\quad\mbox{and}\quad\frac{4p^2}{N(N-2p)}<\frac{(N-2)p-N}{N(p-1)}.
\]

The first restriction holds automatically whenever \(N\ge 6\) and \(p\le\sqrt{N/2}\), as in the previous case; the second is equivalent to
\(G_{N}(p):=4p^{3}+(2N-8)p^{2}-N^{2}p+N^{2}<0\).

Evaluating \(G_{N}\) at the endpoints of the admissible interval one finds
\[
G_{N}\!\left(\frac{N}{N-2}\right)=\dfrac{8N^{2}}{(N-2)^{3}}>0
\]
for every \(N>2\), while
\[
G_{N}\left(\frac{2N-2}{N}\right)=\dfrac{4(N-1)^{3}}{N^{3}}\bigl(N^{2}-8N+8\bigr)
\]
is positive for \(N=6,7\) and negative for all \(N\ge 8\).
  For \(N=6\) and \(N=7\) both endpoint values are positive, and one can show that \(G_{N}(p)\) remains positive throughout \(\bigl(\tfrac{N}{N-2},\tfrac{2N-2}{N}\bigr)\).  For every \(N\ge 8\), however, \(G_{N}\) starts positive at \(p=\tfrac{N}{N-2}\) and ends negative at \(p=\tfrac{2N-2}{N}\). Since \(G_{N}''(p)=24p+4(2N-8)>0\) for every \(p>0\), the polynomial is strictly convex, hence can cross the axis at most once. This guarantees a single root \(p_{0}(N)\) inside the interval, with \(G_{N}(p)<0\) precisely for
\(p_{0}(N)<p<\tfrac{2N-2}{N}\). Thus condition \(G_{N}(p)<0\) holds exactly on that sub-interval when \(N\ge 8\) and on no admissible \(p\) when \(N=6,7\).
\end{itemize}

Putting all the cases together, we finish the proof.
\end{proof}

\end{document}